\documentclass[11pt]{article}

\def\TEXPAD{0}
\def\DARKMODE{0}

\PassOptionsToPackage{obeyspaces}{url}
\usepackage{amsfonts, amsmath, amssymb}
\usepackage{enumitem}

\usepackage[hidelinks]{hyperref}
\usepackage{bookmark}
\bookmarksetup{open,numbered,depth=5} % paragraphs go into outlines

\usepackage{amsthm}

\usepackage[usenames,dvipsnames]{xcolor}
\colorlet{prettygreen}{ForestGreen!60!LimeGreen}
\colorlet{graphorange}{orange}
\if\DARKMODE1
    \colorlet{commentblue}{cyan}
    \colorlet{graphblue}{cyan}
    \colorlet{graphred}{red}
    \colorlet{graphgreen}{prettygreen}
    \colorlet{commentpurple}{yellow}
    \colorlet{graphpurple}{violet!80!white}
    \colorlet{commentorange}{orange}
    \colorlet{graphyellow}{violet}
\else
    \colorlet{commentblue}{cyan!60!blue}
    \colorlet{graphblue}{blue!60!cyan}
    \colorlet{graphred}{red}
    \colorlet{graphgreen}{prettygreen}
    \colorlet{commentpurple}{violet}
    \colorlet{graphpurple}{violet}
    \colorlet{commentorange}{orange!70!red}
    \colorlet{graphyellow}{yellow}
\fi
\colorlet{graphyellow}{yellow}
\usepackage{tikz}
\usetikzlibrary{calc, positioning}
\tikzset{vtx/.style={circle, fill, inner sep=1.5pt}}
\tikzset{openvtx/.style={circle, draw, inner sep=1.5pt}}

\tikzset{trianglecolors/.style={
e1/.style={graphred},
e2/.style={blue},
e3/.style={graphgreen},
e4/.style={graphpurple},
e5/.style={graphyellow},
e6/.style={graphorange}
}}

\tikzset{four colors/.style={
emptyedge/.style={black!60, dashed},
e11/.style={emptyedge}, e22/.style={emptyedge}, e33/.style={emptyedge}, e44/.style={emptyedge}, e55/.style={emptyedge}, e66/.style={emptyedge}, e77/.style={emptyedge}, e88/.style={emptyedge}, e99/.style={emptyedge},
color1/.style={graphred},
color2/.style={graphblue},
color3/.style={graphgreen},
color4/.style={graphyellow},
e12/.style={color1}, e13/.style={color1}, e14/.style={color2}, e15/.style={color4}, e16/.style={color3}, e17/.style={color2}, e18/.style={color3}, e19/.style={color4},
e21/.style={color1}, e23/.style={color1}, e24/.style={color3}, e25/.style={color2}, e26/.style={color4}, e27/.style={color4}, e28/.style={color2}, e29/.style={color3},
e31/.style={color1}, e32/.style={color1}, e34/.style={color4}, e35/.style={color3}, e36/.style={color2}, e37/.style={color3}, e38/.style={color4}, e39/.style={color2},
e41/.style={color2}, e42/.style={color3}, e43/.style={color4}, e45/.style={color1}, e46/.style={color1}, e47/.style={color2}, e48/.style={color4}, e49/.style={color3},
e51/.style={color4}, e52/.style={color2}, e53/.style={color3}, e54/.style={color1}, e56/.style={color1}, e57/.style={color3}, e58/.style={color2}, e59/.style={color4},
e61/.style={color3}, e62/.style={color4}, e63/.style={color2}, e64/.style={color1}, e65/.style={color1}, e67/.style={color4}, e68/.style={color3}, e69/.style={color2},
e71/.style={color2}, e72/.style={color4}, e73/.style={color3}, e74/.style={color2}, e75/.style={color3}, e76/.style={color4}, e78/.style={color1}, e79/.style={color1},
e81/.style={color3}, e82/.style={color2}, e83/.style={color4}, e84/.style={color4}, e85/.style={color2}, e86/.style={color3}, e87/.style={color1}, e89/.style={color1},
e91/.style={color4}, e92/.style={color3}, e93/.style={color2}, e94/.style={color3}, e95/.style={color4}, e96/.style={color2}, e97/.style={color1}, e98/.style={color1},
}}

\usepackage{etoolbox} 
\patchcmd{\thebibliography}{\leftmargin\labelwidth}{\leftmargin\labelwidth\addtolength\itemsep{-0.1\baselineskip}}{}{}

\usepackage{mathtools}

\makeatletter
\newcommand\thankssymb[1]{\textsuperscript{\@fnsymbol{#1}}}
\makeatother

\author{Ting-Wei Chao\thanks{Department of Mathematics, Massachusetts Institute of Technology, Cambridge, MA, USA. Email: {\tt twchao@mit.edu}}, 
Maya Sankar\thanks{Department of Mathematics, Massachusetts Institute of Technology, Cambridge, MA, USA. Email: {\tt mayars@mit.edu}}, 
Hung-Hsun Hans Yu\thanks{Department of Mathematics, Princeton University, Princeton, NJ, USA.  Email: {\tt hansonyu@princeton.edu}}}

\title{On Colorful Kruskal--Katona Theorems}
\date{}

\usepackage[capitalize, nameinlink]{cleveref}

\usepackage{indentfirst}

\newtheorem{theorem}{Theorem}[section]
\newtheorem{lemma}[theorem]{Lemma}
\newtheorem{proposition}[theorem]{Proposition}
\newtheorem{corollary}[theorem]{Corollary}
\newtheorem{claim}[theorem]{Claim}
\newtheorem*{claim*}{Claim}

\theoremstyle{definition}
\newtheorem{definition}[theorem]{Definition}

\newtheorem{question}[theorem]{Question}

\newtheorem{conjecture}[theorem]{Conjecture}

\theoremstyle{remark}
\newtheorem{remark}[theorem]{Remark}

\if\TEXPAD0 % the following line breaks Maya's editor for some reason
	\crefname{claim}{Claim}{Claims}
\fi
\newcommand{\proofofclaim}{\renewcommand*{\qedsymbol}{$\blacksquare$}}

\newcommand*{\eqdef}{\stackrel{\mbox{\normalfont\tiny def}}{=}}   % definition by equality                      
\newcommand*{\PP}{\mathbb{P}}  
\newcommand*{\veps}{\varepsilon}                                 % Nice-looking epsilon
\newcommand*{\RR}{\mathbb{R}}                                     % Real numbers
\newcommand*{\EE}{\mathbb{E}}                                     % Expectation            
\newcommand*{\HH}{\mathbb{H}}  
\newcommand{\abs}[1]{\left\lvert#1\right\rvert}
\DeclarePairedDelimiterX{\infdivx}[2]{(}{)}{#1\;\delimsize\|\;#2}
\newcommand*{\kldiv}{d_{\textup{KL}}\infdivx}
     
\DeclareMathOperator{\supp}{supp}                               % volume 
\newcommand*{\FF}{\mathbb{F}}  
\newcommand*{\cL}{\mathcal{L}}                          
\newcommand{\bfi}{{\mathbf i}}
\newcommand{\bfj}{{\mathbf j}}
\newcommand{\bfk}{{\mathbf k}}
\newcommand{\bfc}{{\mathbf c}}

\newcommand{\NN}{\mathbb N}

\newcommand{\sg}{\sigma}

\newcommand\coin{
\mathrel{\ooalign{\hss$\bigcirc$\hss\cr\kern0.7ex\hbox{\scalebox{0.8}{$\$$}}}}}

\begin{document}

\maketitle

\begin{abstract}
What is the maximum number of rainbow triangles in an edge-colored graph with $m$ edges and $r$ colors? Using entropic techniques, we prove an upper bound of $C_rm^{3/2}$ rainbow triangles with $C_r=\sqrt{\frac{2(r-2)}{9r}}$; this constant is best possible whenever there exists an affine plane of order $r-1$. We also show that constructions attaining at least $(C_r-\varepsilon_r)m^{3/2}$ rainbow triangles must exhibit an affine plane structure, which further improves the upper bound if no such affine plane exists.

We also consider the problem of counting properly edge-colored cliques of larger sizes. Surprisingly, if the number $r$ of colors is odd, this count is instead maximized by blowups of a properly edge-colored $K_{r+1}$.
\end{abstract}

\section{Introduction}

The Kruskal--Katona theorem is one of the fundamental theorems in the field of extremal combinatorics. Independently proven by Kruskal \cite{Kru63} and Katona \cite{Kat68}, the theorem determined the maximum number of cliques of size $s$ in any graph $G$ with $m$ edges. It also generalizes to hypergraphs. A simplified version of the theorem was given by Lov\'asz \cite{Lov93}, stating that in a $k$-uniform hypergraph with $\binom{x}{k}$ edges there are at most $\binom{x}{s}$ cliques of size $s$, for any real number $x\geq s$. There are many extensions of this theorem. For example, one can generalize by replacing cliques with any graph $H$, asking for the maximum number of copies of $H$ in any graph with $m$ edges. The asymptotics of the answer to this question were determined by Alon \cite{Alon81}, and later generalized by Friedgut and Kahn \cite{FK98} to hypergraphs. 

Another way to generalize the Kruskal--Katona theorem is to consider edge-colored graphs.
We say that a subgraph in an edge-colored graph is \emph{rainbow} if all of its edges have different colors. Inspired by the joints problem from incidence geometry, the first and the third author \cite{CY24kk} studied the following rainbow version of the Kruskal--Katona theorem, which they called the \emph{rainbow triangle problem}. 
If $G$ is a $3$-edge-colored simple graph with a given number of edges, what is the maximum number of rainbow triangles in $G$? 
The first and the third author proved the following result.
\begin{theorem}[\cite{CY24kk}]\label{thm:rainbow-triangle}
    In a simple 3-edge-colored graph with $m_r$ red edges, $m_g$ green edges, and $m_b$ blue edges, the number $T$ of rainbow triangles satisfies
    \[T^2\leq 2m_rm_gm_b.\]
\end{theorem}
This bound is tight in certain regimes, as shown by the following example. Partition the edge set of $K_4$ into three perfect matchings, each forming a different color class. In this colored graph, we have $m_r=m_g=m_b=2$ and $T=4$, so that $T^2=2m_rm_gm_b$. 
Any balanced blowup of this colored graph also yields a construction with $m_r=m_g=m_b$ for which the bound is tight. It remains open whether the bound is tight if the sizes of the color classes are unbalanced.
Note that, by the AM--GM inequality, any $3$-edge-colored simple graph with $m$ edges contains at most $\sqrt{\frac{2}{27}}m^{3/2}$ rainbow triangles.

The first main result of this paper is to generalize the previous result to more than three colors.
\begin{theorem}\label{thm:r-color}
Let $r\geq 3$ and $m$ be positive integers.
Let $G$ be an $r$-edge-colored simple graph with $m$ edges in total. Then the number of rainbow triangles is at most $\sqrt{\frac{2(r-2)}{9r}}m^{3/2}$.
\end{theorem}

This bound is tight if $q=r-1$ is a prime power. Let $G$ be a complete graph with vertex set $\mathbb F_q^2$ and let the color of an edge between two points $p_1,p_2\in V(G)$ correspond to the slope of the line from $p_1$ to $p_2$ in $\mathbb F_q^2$. Because $G$ is a complete graph, we have $m=\binom{q^2}2=\frac 12q^2(q^2-1)$. Moreover, any three non-collinear points in $\mathbb F_q^2$ form a rainbow triangle, so the number of rainbow triangles in $G$ is
\[
\frac 16q^2(q^2-1)(q^2-q)=\frac 16q^3(q-1)^2(q+1)=\sqrt{\frac{2(q-1)}{9(q+1)}}m^{3/2},
\]
matching the bound in \cref{thm:r-color}. Taking balanced blowups of $G$ shows that \cref{thm:r-color} is tight for infinitely many $m$.

More generally, \cref{thm:r-color} is tight if there is a projective plane of order $q=r-1$. To see this, let $P$ be such a projective plane and fix a line $\ell$ in $P$. We think of $\ell$ as the ``line at infinity'' and let $A=P\setminus\ell$ be the ``affine part'' of $P$, corresponding to the affine plane $\mathbb F_q^2$ in the previous construction. Let $G$ be a complete graph on vertex set $A$ colored by $r$ colors corresponding to the $r=q+1$ points of $\ell$. For any two points $p_1,p_2\in A$, we color the edge $p_1p_2$ with the color corresponding to $\overline{p_1p_2}\cap\ell$. The graph $G$ has $m=\binom{|A|}2$ edges and any three non-collinear points in $A$ form a rainbow triangle, so the number of rainbow triangles is $\frac 16|A|(|A|-1)(|A|-q)$. These quantities match the computation above because $|A|=|P|-|\ell|=(q^2+q+1)-(q+1)=q^2$. Thus, this graph $G$ and its balanced blowups are also tight examples for \cref{thm:r-color}.

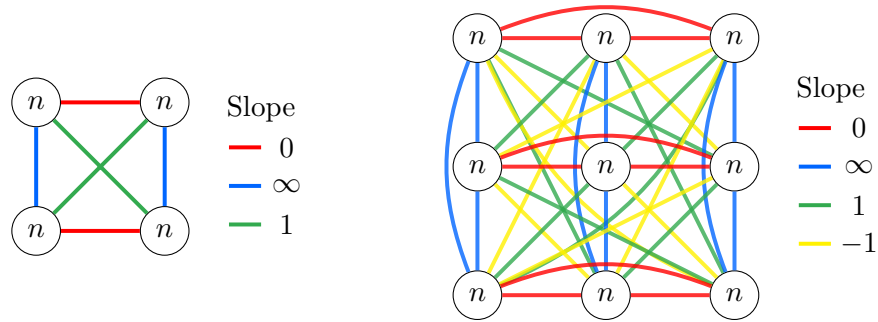
\begin{figure}[h]
\centering
\begin{tikzpicture}[four colors, scale=1.7, baseline={(0,0)}]
\foreach \i in {1,...,4} {
	\draw let \n1 = {mod(\i-1,2)} in 
		(\n1, {-(\i-\n1-1)/2+0.5}) node[circle,draw=black] (v\i) {$n$};
}
\path (v2) --++ (0.5,-0.05) 
	node[anchor=west, inner sep=-1pt] (s0) {Slope}
	--++ (0,-0.3) coordinate (s1)
	--++ (0,-0.3) coordinate (s2)
	--++ (0,-0.3) coordinate (s3);
	
\begin{scope}[line width=1.5pt]
	\draw[color1] (v1) -- (v2) (v3) -- (v4);
	\draw[color2] (v1) -- (v3) (v2) -- (v4);
	\draw[color3] (v1) -- (v4) (v2) -- (v3);
	\foreach \i/\j in {1/0,2/\infty,3/1} {
	\draw[color\i] (s\i) --++ (0.25,0) ++ (0.2,0) node[black] {$\j$};
	}
\end{scope}
\end{tikzpicture}
\qquad\qquad
\begin{tikzpicture}[four colors, scale=1.7, baseline={(0,0)}]
\foreach \i in {1,...,9} {
	\draw let \n1 = {mod(\i-1,3)} in 
		(\n1, {-(\i-\n1-1)/3+1}) node[circle,draw=black] (v\i) {$n$};
}
\foreach \i [evaluate=\i as \ii using int(\i+1)] in {1,...,8} {
	\foreach \j in {\ii,...,9} {
		\draw[e\i\j, line width=1.5pt, opacity=0.8] let 
			\n1={(\j-\i==2 && mod(\i,3)==1) || (\i==3 && \j==7) ? 21 : (
				(\j-\i == 6) || (\i==1 && \j==9) ? -21 : 0
			)} 
		in(v\i) to[bend left=\n1] (v\j);
	}
}
\path (v3) --++ (0.5,-0.4) 
	node[anchor=west, inner sep=-1pt] (s0) {Slope}
	--++ (0,-0.3) coordinate (s1)
	--++ (0,-0.3) coordinate (s2)
	--++ (0,-0.3) coordinate (s3)
	--++ (0,-0.3) coordinate (s4);
\foreach \i/\j in {1/0,2/\infty,3/1,4/-1} {
	\draw[color\i, line width=1.5pt] (s\i) --++ (0.25,0) ++ (0.22,0) node[black] {$\j$};
}
\end{tikzpicture}

\caption{Constructions for \cref{thm:r-color} when $r=3$ (left) and $r=4$ (right), which are balanced blowups of the graphs induced by affine planes of order $2$ or $3$.}
\label{fig:affine-plane-coloring}
\end{figure}

For our convenience, we will call the graph $G$ obtained from this construction \emph{the graph induced by an affine plane of order $(r-1)$.}
An equivalent characterization of $G$ is that it is an $r$-edge-colored complete graph of $(r-1)^2$ vertices so that each color class is a union of $(r-1)$ vertex-disjoint $(r-1)$-cliques.

\subsection{Comparison with previous work}
Our proof of \cref{thm:r-color} is based on a different proof of \cref{thm:rainbow-triangle}.
The original proof in \cite{CY24kk} is a combination of the entropy method and a counting argument, and later a purely entropic proof was given in \cite{CY24rt} using an injection as a key ingredient. 
Recently, Balogh, Bradshaw, Garcia, and Lidick\'y \cite{BBGL25} gave a human-readable flag algebra proof to \cref{thm:rainbow-triangle}, which can be reformulated as either a counting proof or an entropic proof.
Their proof also relies on a key injection, which is significantly simpler than the one in \cite{CY24rt}.
They also characterized the optimal configurations, showing that they are blowups of the colored $K_4$ construction above, possibly with some isolated vertices.

In comparison, our proof uses the mixture bound, which was introduced in \cite{CY24kk,CY24turan}.
Informally, the mixture bound allows us to upper bound the entropies of random variables supported on disjoint sets, in terms of the entropy of a new random variable that ``mixes'' their distributions.
With some carefully constructed random variables and simple applications of the mixture bound, we can prove the following statement that implies both \cref{thm:rainbow-triangle} and \cref{thm:r-color}.
\begin{theorem}\label{thm:entropy-r-color}
Let $G$ be an $r$-edge-colored simple graph and let $(X_1,X_2,X_3)$ be any symmetric distribution over rainbow triangles. That is, $(X_1,X_2,X_3)$ has the same distribution as $(X_{\sigma(1)},X_{\sigma(2)},X_{\sigma(3)})$ for any permutation $\sigma$ of $\{1,2,3\}$. Then 
\[2\HH( X_1,X_2,X_3)\leq 3\HH(X_1,X_2)+\log_2\left(\frac{r-2}{r}\right).\] 
\end{theorem}
We remark that the proof is injection-free, which leads us to believe that the proof is genuinely different from the previous proofs.
\cref{thm:r-color} follows immediately from \cref{thm:entropy-r-color} by choosing $(X_1,X_2,X_3)$ uniformly at random and by the uniform bound. The details are given in \cref{sec:many_color_RT}. 
Additionally, we will derive \cref{thm:rainbow-triangle} from the $r=3$ version of \cref{thm:entropy-r-color} in \cref{subsec:two_tricks}.

\subsection{Optimal configurations and stability}
We have seen that the affine plane of order $(r-1)$ gives a construction that is tight for \cref{thm:r-color}.
We can also show that this is essentially the only possible extremal construction.
\begin{theorem}\label{thm:tight-construction}
    Let $r\geq 3$ and $m$ be positive integers.
    If $G$ is a simple $r$-edge-colored graph with $m$ edges and exactly $\sqrt{\frac{2(r-2)}{9r}}m^{3/2}$ rainbow triangles, then $G$ is a balanced blowup of the graph induced by an affine plane of order $(r-1)$ possibly with some additional isolated vertices.
\end{theorem}
In fact, we can strengthen this to a stability result for \cref{thm:entropy-r-color}.
However, as it is technical to state, we postpone the full statement to \cref{thm:strong-stability} and only state a rough statement here.
\begin{theorem}\label{thm:stability_brief}
    Under the same setting as \cref{thm:entropy-r-color}, an entropic stability version of \cref{thm:tight-construction} holds when 
    \[2\HH(X_1,X_2,X_3)> 3\HH(X_1,X_2)+\log_2\left(\frac{r-2}{r}\right)-\alpha r^{-16}\]
    for some absolute constant $\alpha>0$ small enough.
\end{theorem}
The full stability result immediately implies that the bound in \cref{thm:r-color} can be improved by some constant depending on $r$ if there is no finite projective plane of order $(r-1)$.
\begin{corollary}\label{cor:weak-stability}
    Let $r\geq 3$ and $m$ be positive integers.
    Suppose that there exists no finite projective plane of order $r-1$.
    Then there exists some $\varepsilon =\Theta(r^{-16})$ so that any simple $r$-edge-colored graph $G$ with $m$ edges contains at most $\left(\sqrt{\frac{2(r-2)}{9r}}-\varepsilon\right)m^{3/2}$ rainbow triangles.
\end{corollary}
We make no attempt to optimize the saving $\varepsilon =\Theta(r^{-16})$ and believe this is not the optimal dependency.

We remark that deciding whether a finite projective plane of a given order exists is still a major open question.
Although finite projective planes of prime power order exist by taking projective planes over finite fields, it is even unknown whether a finite projective plane of order $12$ exists.
However, it is known that finite projective planes of order $6$ or $10$ do not exist \cite{Tar1900,Tar1901,Bose38,BR49,LTS89}.
It might thus be interesting to determine the optimal bound in \cref{thm:r-color} when $r=7$ or $11$.

\subsection{Properly colored cliques}
As mentioned above, given the number of edges, the $3$-edge-colored graphs that maximize the number of rainbow triangles are the blowups of $K_4$ colored in a specific way.
This coloring is in fact the unique proper $3$-edge-coloring of $K_4$.
Motivated by this, Balogh, Bradshaw, Garcia, and Lidick\'y \cite{BBGL25} proved the following.
\begin{theorem}[\cite{BBGL25}]
    In a simple graph with $m_r$ red edges, $m_g$ green edges, and $m_b$ blue edges, the number of properly colored $K_4$'s is at most $\frac{1}{4}(m_rm_gm_b)^{2/3}$. 
    Moreover, equality holds if and only if the graph is a balanced blowup of a properly colored $K_4$, possibly with some isolated vertices.
\end{theorem}
Given this, it might be tempting to guess that for any given number of colors, the affine-plane construction also maximizes the number of properly colored $K_4$'s.
To our surprise, this is not the case: if the number $r$ of colors is odd, then the optimizers turn out to be balanced blowups of properly colored $K_{r+1}$.
In fact, this holds when maximizing the number of properly colored $K_d$'s for any $d\geq 4$.
\begin{theorem}\label{thm:properly-colored-Kd-in-r-colors}
    Let $d\geq 4$ and $r\geq d-1$ be integers. Let $G$ be an $r$-edge-colored simple graph with $m$ edges in total. Then the number of properly colored $K_d$'s in $G$ is at most $\binom{r+1}{d}\left(\frac{m}{\binom{r+1}{2}}\right)^{d/2}$.
\end{theorem}
Similar to the proof of \cref{thm:entropy-r-color}, the proof of \cref{thm:properly-colored-Kd-in-r-colors} uses the mixture bound, and there is an entropic formulation (see \cref{thm:entropy-properly-colored-Kd-in-r-colors}) of the statement as well that we will introduce in the proof. 

\subsection{Connection to the joints problem}
As mentioned earlier, the rainbow triangle problem was originally motivated by the joints problem.
Here, we give a brief overview of the connection, and end the discussion with what \cref{thm:r-color} corresponds to in this picture.

The joints problem asks for the maximum number of joints that $m$ lines can form in $\RR^3$, where a \emph{joint} is a common intersection of three non-coplanar lines.
This problem was first studied by Chazelle et al.\ \cite{CEGPSSS92}. 
Attention was brought to this problem later by Wolff \cite{Wol99}, who suggested that this is one of the discrete analogs of the Kakeya problem, a central problem in harmonic analysis. Using the polynomial method developed by Dvir \cite{Dvir09}, Guth and Katz \cite{GK10} showed that the number of joints is $O(m^{3/2})$, which is tight up to the multiplicative constant. Recently, the correct multiplicative constant was determined by Zhao and the third author \cite{YZ23}, and later a tight upper bound was proven by the first and the third author \cite{CY23}. 
In particular, they showed that, for any real number $x\geq 3$, any $\binom{x}{2}$ lines in $\RR^3$ determine at most $\binom{x}{3}$ joints.

The aforementioned result in fact implies Lov\'asz's version of the Kruskal--Katona theorem by restricting to \emph{generically induced configurations}.
In $\RR^3$, a generically induced configuration given by a graph $G=([n],E)$ is a collection of lines $\cL$ obtained as follows. Let $H_1,\dots,H_n$ be planes in $\RR^3$ in general position, and let $\cL=\{H_i\cap H_j\mid ij\in E\}$. Note that each line in $\cL$ corresponds to an edge in $G$, and each joint formed by $\cL$ corresponds to a triangle in $G$. This shows that Lov\'asz's version of the Kruskal--Katona theorem is indeed a corollary of \cite{CY23}.

Many other generalizations of the joints problem have also been considered. One such generalization is the multijoints problem, which is a multicolor analogue of the original problem. Given three sets of lines $\cL_r,\cL_g,\cL_b$, a \emph{multijoint} is a joint formed by one line from each set. 
The multijoints problem, first proposed by Carbery, asks to determine the maximum number of multijoints given $\abs{\cL_r},\abs{\cL_g},\abs{\cL_b}$. Using the polynomial method with higher order derivatives, Zhang \cite{Zha20} showed that the number of multijoints is $O(\sqrt{\abs{\cL_r}\abs{\cL_g}\abs{\cL_b}})$, which is tight up to the multiplicative constant. The upper bound was improved to $\sqrt{6\abs{\cL_r}\abs{\cL_g}\abs{\cL_b}}$ by Tidor, Zhao, and the third author \cite{TYZ22}.
However, the multiplicative constant is likely not optimal for the following reason.
When restricted to generically induced configurations, the problem becomes exactly the rainbow triangle problem, which by \cref{thm:rainbow-triangle} has a better upper bound $\sqrt{2\abs{\cL_r}\abs{\cL_g}\abs{\cL_b}}$.
One of the motivations to consider the rainbow triangle problem in \cite{CY24kk} is to see whether the proof for the rainbow triangle theorem can be lifted up to an improvement on the multijoints problem, although it is still currently open.

If instead of three sets of lines, we have $r$ sets of lines $\cL_1,\cL_2,\ldots, \cL_r$, we can still define a multijoint to be a joint formed by three lines coming from three different sets.
What \cref{thm:r-color} implies is that for generically induced configurations, $r$ sets of lines form at most $
\sqrt{\frac{2(r-2)}{9r}}m^{3/2}$ multijoints if $m$ is the total number of lines.
Deciding whether this remains true for all joints configurations seems like an interesting problem, although we believe that it would be hard to prove if true.

There is also a corresponding statement of \cref{thm:properly-colored-Kd-in-r-colors} in the setting of joints, although the statement is not as natural in that setting.
We thus omit the discussion here and refer the readers to \cite{CY23,CY24kk} for more details on how to relate graph-theoretic statements to statements about joints.
From the point of view of joints, counting rainbow $K_d$'s is more natural.
However, we believe that this is a much more difficult problem, and we will discuss this briefly in \cref{subsec:rainbow_Kd}.

\subsection{Structure of the paper}
We begin with a quick review of the definitions and propositions of Shannon entropy in \cref{sec:prelim}. In \cref{sec:new_proof_2rgb}, we present another new proof of \cref{thm:rainbow-triangle}, using it to motivate some important techniques we rely on heavily throughout the paper. Then we prove \cref{thm:r-color,thm:entropy-r-color} in \cref{sec:many_color_RT} and \cref{thm:properly-colored-Kd-in-r-colors} in \cref{sec:proper_Kd}. In \cref{sec:sturctural_stability}, we prove \cref{thm:tight-construction}, introduce the notions and some properties of KL divergence and total variation distance, and then present the proof of \cref{thm:stability_brief,cor:weak-stability}. Finally, we make some concluding remarks in \cref{sec:remark}.

\paragraph{AI Statement}
No AI tools were used in the preparation of this project, including in the research process and the writing of this manuscript.

\section{Preliminaries}\label{sec:prelim}
\subsection{Shannon entropy}
In this subsection, we include some definitions and propositions we need about Shannon entropy. For a more detailed introduction and proofs for the propositions in this section, see \cite{AS00}.

For any discrete random variable $X$, we denote by $P_X(x)$ the probability $\PP(X=x)$. Throughout the paper, all random variables $X$ will be discrete and will have finite support. In other words, $\supp(X)\eqdef\{x\mid P_X(x)>0\}$ is finite.
\begin{definition}[Entropy]
    For any random variable $X$, its \emph{entropy} is defined as
    \[\HH(X)\eqdef\quad\smashoperator{\sum_{x\in \supp(X)}}-P_X(x)\log_2P_X(x).\]
    For any $d$-tuple $(X_1,\dots,X_d)$ of random variables, we use $\HH(X_1,\dots,X_d)$ to denote the entropy of the tuple, viewing it as a random variable with support contained in $\supp(X_1)\times\dots\times\supp(X_d)$.
\end{definition}

The following uniform bound helps us connects the entropy of $X$ and the size of $\supp(X)$.
\begin{proposition}[Uniform bound]\label{prop:uniform}
    Let $X$ be a random variable. We have
    \[\HH(X)\leq \log_2\abs{\supp(X)},\]
    where the equality holds if and only if $X$ is uniform.
\end{proposition}

We also define the conditional entropy of one random variable conditioned on another.
\begin{definition}[Conditional entropy]
    For any random variables $X,Y$, we define
    \[\HH(Y\mid X)\eqdef\HH(X,Y)-\HH(X).\]
    Equivalently, we have
    \[\HH(Y\mid X)=\quad\smashoperator{\sum_{x\in\supp(X)}}P_X(x)\HH(Y\mid X=x)=\quad\smashoperator{\sum_{(x,y)\in\supp(X,Y)}}-P_{X,Y}(x,y)\log_2\left(\frac{P_{X,Y}(x,y)}{P_X(x)}\right).\]
\end{definition}

Using conditional entropy, we can write the entropy of a tuple as a telescoping sum, which we call the chain rule.
\begin{proposition}[Chain rule]
    Let $X_1,\dots,X_d$ be random variables. We have
    \[\HH(X_1,\dots,X_d)=\HH(X_1)+\HH(X_2\mid X_1)+\dots+\HH(X_d\mid X_1,\dots,X_{d-1}).\]
\end{proposition}

We also need the subadditivity and the dropping conditions inequalities of entropy.
\begin{proposition}[Subadditivity]\label{prop:subadd}
    Let $X,Y,Z$ be random variables. We have
    \[\HH(X,Y)\leq \HH(X)+\HH(Y),\]
    with equality if and only if $X,Y$ are independent. We also have
    \[\HH(X,Y\mid Z)\leq \HH(X\mid Z)+\HH(Y\mid Z),\]
    with equality if and only if $X,Y$ are conditionally independent given $Z$.
\end{proposition}
\begin{proposition}[Dropping conditions]\label{prop:drop}
    Let $X,Y,Z$ be random variables. We have
    \[\HH(X\mid Y)\leq \HH(X),\]
    with equality if and only if $X,Y$ are independent. We also have
    \[\HH(X\mid Y, Z)\leq \HH(X\mid Z),\]
    with equality if and only if $X,Y$ are conditionally independent given $Z$.
\end{proposition}

We also need the following fact. 
\begin{proposition}\label{prop:recover}
    Let $f$ be a deterministic function and let $X$ be a random variable. We have
    \[\HH(f(X)\mid X)=0.\]
\end{proposition}

\section{New proof for three colors}\label{sec:new_proof_2rgb}

To illustrate the techniques used in this paper, we start by proving \cref{thm:rainbow-triangle-simple}, which corresponds to \cref{thm:rainbow-triangle} in the case $m_r=m_g=m_b$. In \cref{subsec:two_tricks}, we highlight several key ideas in the proof. Building them out contributes yet another proof of \cref{thm:rainbow-triangle} to the literature. Along the way, we prove the $r=3$ case of \cref{thm:entropy-r-color}. Although some ideas are shared, let us remark that the argument in \cref{subsec:two_tricks} is not a specialization of the proof of \cref{thm:entropy-r-color} presented in \cref{sec:many_color_RT}; the latter is more complex.

\subsection{Simple proof}\label{subsec:simple_proof}
To motivate the techniques used in this paper, we start by proving the following weakening of \cref{thm:rainbow-triangle}.
As remarked in the introduction, \cref{thm:rainbow-triangle-simple} is an immediate corollary of \cref{thm:rainbow-triangle} by the AM--GM inequality.

\begin{theorem}[\cite{CY24kk}]\label{thm:rainbow-triangle-simple}
Let $G$ be a simple graph with $m$ edges together with a 3-edge-coloring $\chi:E(G)\to \{\mathrm{red},\mathrm{green},\mathrm{blue}\}$. Then $G$ contains at most $\sqrt{\frac{2}{27}}m^{3/2}$ rainbow triangles.
\end{theorem}

\begin{proof}
Let $(X_1,X_2,X_3)$ be a uniformly random triple of vertices forming a rainbow triangle in $G$. Throughout this proof, subscripts are taken modulo 3. For each $i\in[3]$, sample $Y_i$ uniformly among those vertices for which $(X_{i-1}, Y_i, X_{i+1})$ is a rainbow triangle with the same coloring as $(X_{i-1},X_i,X_{i+1})$. In fact we need only require that $(X_{i-1}, Y_i, X_{i+1})$ is a rainbow triangle satisfying $\chi(X_{i-1}Y_i)=\chi(X_{i-1}X_i)$, as this implies equality of the last color $\chi(Y_iX_{i+1})=\chi(X_iX_{i+1})$ as well. Lastly, sample $Z_{i+1}$ uniformly among those vertices for which $(X_{i-1},Y_i,Z_{i+1})$ forms a rainbow triangle.

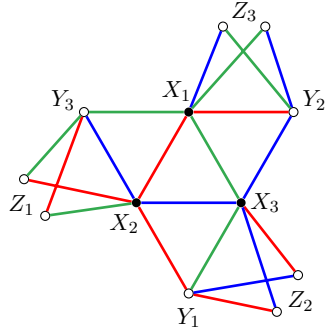
\begin{figure}[h]	
\centering
\begin{tikzpicture}[trianglecolors,scale=0.8, every node/.append style={scale=0.8}]
\foreach \i in {1,2,3} {
	\path let \n1={int(\i+0)},\n2={int(mod(\i,3)+1)},\n3={int(mod(\i+1,3)+1)},
			\n{sqrt3}={sqrt(3)} in
	(120*\i-30:1) coordinate[vtx] (X\n1) 
		--+ (120*\i:0.4) node{$X_\n1$}
		--++ (120*\i-120:\n{sqrt3}) coordinate[openvtx] (Y\n2)
		--+ (120*\i-90:0.4) node{$Y_\n2$}
		--++ (120*\i:\n{sqrt3}-0.1) --+(120*\i-30:0.25) node{$Z_\n3$}
		--+ (120*\i+60:0.35) coordinate[openvtx] (Z\n31)
		--+ (120*\i-120:0.35) coordinate[openvtx] (Z\n32);
}

\foreach \i in {1,2,3} {
	\draw[e\i, line width=1pt] let \n1={int(\i+0)},\n2={int(mod(\i,3)+1)},\n3={int(mod(\i+1,3)+1)} in
	(Y\n2) -- (X\n1) -- (X\n2) -- (Y\n1) 
	(X\n2) -- (Z\n12) 
	(Y\n3) -- (Z\n11)
	(Y\n1) -- (Z\n22)
	(X\n3) -- (Z\n21);
}
\end{tikzpicture}
\caption{The possible colorings of edges to $Y_i$ and $Z_{i+1}$ given a fixed rainbow triangle $(X_1,X_2,X_3)$.}
\label{fig:coloring-id-ref}
\end{figure}

We first compute the conditional entropies of $Y_2$ and $Z_3$ given $X_1,X_2,X_3$. Since $Y_2$ and $X_2$ are independently sampled from the same distribution once $X_1,X_3,\chi(X_1X_2)$ are fixed, we know that
\[\HH(Y_2\mid X_1,X_2,X_3) = \HH(Y_2\mid X_1,X_3,\chi(X_1X_2))= \HH(X_2\mid X_1,X_3,\chi(X_1X_2)).\]
Using the chain rule twice, the right side expands as
\[\HH(X_2\mid X_1,X_3)+\HH(\chi(X_1X_2)\mid X_1,X_2,X_3)-\HH(\chi(X_1X_2)\mid X_1,X_3).\]
As $\chi(X_1X_2)$ is determined by $X_1,X_2,X_3$, \cref{prop:recover} implies $\HH(\chi(X_1X_2)\mid X_1,X_2,X_3)=0$.
Moreover, by the uniform bound, we have that $\HH(\chi(X_1X_2)\mid X_1,X_3)\leq \log_22=1$ as $\chi(X_1X_2)$ is in $ [3]\backslash \{\chi(X_1X_3)\}$, and this set has size 2.
Therefore,
\[\HH(Y_2\mid X_1,X_2,X_3)\geq \HH(X_2\mid X_1,X_3)-1.\]

Next, observe that $\HH(Z_3\mid X_1,X_2,X_3,Y_2) = \HH(Z_3\mid X_1,Y_2)$, because the choice of $Z_3$ conditioned on $X_1,Y_2$ is independent of the choices of $X_2$ and $X_3$. To further understand this quantity, first observe that $(X_1,Y_2,X_3)$ is uniformly distributed among all rainbow triangles, as $Y_2$ and $X_2$ have the same distribution conditioned on $X_1,X_3$. It follows that $(X_1,Y_2,Z_3)$ is also uniformly distributed among all rainbow triangles, as $Z_3$ and $X_3$ have the same distribution conditioned on $X_1,Y_2$. Thus,
\[
\HH(Z_3\mid X_1,X_2,X_3,Y_2) = \HH(Z_3\mid X_1,Y_2) = \HH(X_3\mid X_1,Y_2).
\]
Therefore, by the chain rule,
\begin{align*}
\HH(Y_2,Z_3\mid X_1,X_2,X_3) 
&= \HH(Y_2\mid X_1,X_2,X_3)+\HH(Z_3\mid X_1,X_2,X_3,Y_2)
\\&\geq \HH(X_2\mid X_1,X_3)+\HH(X_3\mid X_1,X_2)-1
\\&=2\HH(X_3\mid X_1,X_2)-1\quad\text{by symmetry.}
\end{align*}
By symmetry, $\HH(Y_i,Z_{i+1}\mid X_1,X_2,X_3)\geq 2\HH(X_3\mid X_1,X_2)-1$ for each $i\in[3]$.

Now, sample $\bfi\in[3],\bfj\in[2]$ independently and uniformly at random and let
\[(A,B)=\begin{cases}(Y_\bfi, Z_{\bfi+1}) & \textup{if }\bfj=1,\\(Z_{\bfi+1},Y_\bfi)&\textup{if }\bfj=2.\end{cases}\] The remainder of our proof hinges on the following claim, which shows that $\bfi,\bfj, Y_\bfi, Z_{\bfi+1}$ can be recovered from $X_1,X_2,X_3,A,B$.

\begin{claim}\label{claim:supports-id-ref-2}
The six random quintuples $(X_1,X_2,X_3,Y_i,Z_{i+1})$ for $i\in[3]$ and $(X_1,X_2,X_3,Z_{i+1},Y_i)$ for $i\in [3]$ have pairwise disjoint supports.
\end{claim}

\begin{proof}\proofofclaim
Suppose that $(x_1,x_2,x_3,a,b)$ is in the support of one of the six random quintuples.
We will show that it cannot be in the supports of the five other random quintuples.

Without loss of generality, assume that $(x_1,x_2,x_3,a,b)\in \supp(X_1,X_2,X_3,Y_1,Z_2)$.
Then we have $\chi(ax_2) = \chi(x_1x_2)$ and $\chi(ax_3) = \chi(x_1x_3)$.
This shows that conditioned on $X_1=x_1,X_2=x_2,X_3=x_3$, the vertex $a$ cannot be in the supports of $Y_2$, $Y_3$, or $Z_2$.
Therefore we just need to show that $(x_1,x_2,x_3,a,b)\not\in \supp(X_1,X_2,X_3,Z_3,Y_2)\cup \supp(X_1,X_2,X_3,Z_1,Y_3)$.

If $(x_1,x_2,x_3,a,b)\in \supp(X_1,X_2,X_3,Z_3,Y_2)$, then $\chi(bx_1) = \chi(x_1x_2)$ and $\chi(bx_3) = \chi(x_2x_3)$.
Because $\chi(ax_3) = \chi(x_1x_3)$ and $x_3ab$ is a rainbow triangle, this forces $\chi(ab) = \chi(x_1x_2)$.
This contradicts the fact that $x_1ab$ is a rainbow triangle and that $\chi(bx_1) = \chi(x_1x_2)$. Thus, $\supp(X_1,X_2,X_3,Y_1,Z_2)$ and $\supp(X_1,X_2,X_3,Z_3,Y_2)$ are disjoint. 

The case where $(x_1,x_2,x_3,a,b)$ is in the support of $(X_1,X_2,X_3,Z_1,Y_3)$ is analogous by symmetry, as the same argument shows that $\supp(X_1,X_2,X_3,Y_3,Z_1)\cap\supp(X_1,X_2,X_3,Z_2,Y_1)=\emptyset$.
\qedhere
\end{proof}

Since $X_1,X_2,X_3,A,B$ can be used to recover $\bfi,\bfj, Y_\bfi, Z_{\bfi+1}$, by \cref{prop:recover}, we have
\[\HH(\bfi,\bfj, Y_\bfi, Z_{\bfi+1}\mid A,B,X_1,X_2,X_3)=0,\]
and so by the chain rule,
\[\HH(A,B\mid X_1,X_2,X_3)=\HH(\bfi,\bfj, Y_\bfi, Z_{\bfi+1},A,B\mid X_1,X_2,X_3)\geq \HH(\bfi,\bfj, Y_\bfi, Z_{\bfi+1}\mid X_1,X_2,X_3).\]
We remark that this is actually an equality, as $A,B$ can also be recovered from $X_1,X_2,X_3,\bfi,\bfj,Y_{\bfi},Z_{\bfi+1}$.
The right side can be further expanded as 
\[
\HH(\bfi,\bfj\mid X_1,X_2,X_3)+\HH(Y_\bfi,Z_{\bfi+1}\mid X_1,X_2,X_3,\bfi)
\]
by the chain rule and the fact that $(Y_{\bfi},Z_{\bfi+1})$ is independent of $\bfj$.
The first term is $\log_26$, and the second term is the average of the three entropies $\HH(Y_i,Z_{i+1}\mid X_1,X_2,X_3)$ for $i\in[3]$. Thus,
\[
\HH(A,B\mid X_1,X_2,X_3)=\log_26+\frac 13\sum_{i\in[3]}\HH(Y_i,Z_{i+1}\mid X_1,X_2,X_3)\geq\log_26 + 2\HH(X_3\mid X_1,X_2)-1.
\]
Note that $\{A,B\}$ is always an edge in $E(G)$, showing that the support of the ordered tuple $(A,B)$ has size at most $2m$.
Dropping conditions and applying the uniform bound, we have that
\[
\log_2 2m\geq\log_26+2\HH(X_3\mid X_1,X_2)-1.
\]
Applying the chain rule and the uniform bound $\HH(X_1,X_2)\leq \log_2(2m)$, we conclude that
\[
\HH(X_1,X_2,X_3)=\HH(X_1,X_2)+\HH(X_3\mid X_1,X_2)\leq\log_2(2m)+\frac 12\log_2(2m/3).
\]
Because the number $T$ of rainbow triangles satisfies $\log(6T)=\HH(X_1,X_2,X_3)$ --- as each rainbow triangle can be labeled in exactly six ways --- we derive that $6T\leq (2m)\left(\frac{2m}{3}\right)^{1/2}$. Hence, $T\leq\sqrt{\frac{2}{27}}m^{3/2}$ as desired.
\end{proof}

\begin{remark}
One could simplify the above proof by sampling $\bfi$ but not $\bfj$, and then studying the entropy of the random edge $e=\{Y_\bfi,Z_{\bfi+1}\}$. The analysis from \cref{claim:supports-id-ref-2} implies that the tuple $(\bfi,Y_\bfi,Z_{\bfi+1})$ can be recovered from the unordered pair $e=\{Y_\bfi,Z_{\bfi+1}\}$, and dropping conditions together with a uniform bound show that $\HH(e\mid X_1,X_2,X_3)\leq \HH(e)\leq\log m$. We chose to present the proof in terms of ordered tuples as it provides better motivation for the distributions used later in the paper, where considering ordered tuples is genuinely necessary. 
\end{remark}

\subsection{Two tricks: resampling and mixtures}\label{subsec:two_tricks}

Let us highlight two tricks used in the previous proof that are crucial throughout this paper. The first trick is \emph{resampling random variables}: given several random variables, we can always resample one variable conditioned on one or more of the others. In the above proof, we used that $Y_i$ is distributed according to $X_i$ given $X_{i-1}$, $X_{i+1}$, and $\chi(X_{i-1}Y_i)$, and that $Z_{i+1}$ is distributed according to $X_{i+1}$ given $X_{i-1}$ and $Y_i$. In fact, if $(X_1,X_2,X_3)$ were a rainbow triangle chosen according to some non-uniform distribution, most of our computations would still hold true. In this case, we would choose $Y_i$ according to the distribution of $X_i$ conditioned on $X_{i-1},X_{i+1},\chi(X_{i-1}Y_i)$ and choose $Z_{i+1}$ according to the distribution of $X_{i+1}$ conditioned on $X_{i-1}$.

The second trick is \emph{taking mixtures} of random variables (or tuples of random variables) with disjoint supports. A \emph{mixture} of some (possibly codependent) random variables $A_1,\ldots,A_k$ is a new random variable $B$ sampled by independently choosing an index $\bfi\in[k]$ according to some (not necessarily uniform) distribution and setting $B=A_\bfi$. If the sets $\supp(A_i)$ are pairwise disjoint, then $\bfi$ can be uniquely determined from $B$ and thus $\HH(\bfi\mid B)=0$. More generally, if these supports are \emph{$(s+1)$-wise disjoint}, meaning that any $a\in \bigcup_{i\in[k]}\supp(A_i)$ is in at most $s$ of the sets $\supp(A_i)$, then each possibility for $B$ allows in at most $s$ choices for $\bfi$, and the uniform bound implies $\HH(\bfi\mid B)\leq\log_2s$.
If, in addition, $\bfi\in[k]$ is uniformly distributed then
\[
\HH(B) = \HH(\bfi) + \HH(B\mid \bfi) - \HH(\bfi\mid B)\geq\log_2k + \frac 1k\sum_{i=1}^k\HH(A_i)-\log_2s.
\]

When the entropies $\HH(A_1),\ldots,\HH(A_k)$ are not all the same, a carefully chosen non-uniform distribution on $\bfi$ yields the following stronger bound, due to the first and third authors \cite{CY24turan}. The version stated here is slightly stronger, giving explicit equality when $A_1,\ldots,A_k$ have disjoint supports, which will be necessary to our stability arguments. 

\begin{lemma}[Mixture bound \cite{CY24turan}]\label{lem:mix}
Let $A_1,\ldots,A_k$ be random variables whose supports are $(s+1)$-wise disjoint. Then there is a mixture $B$ of $A_1,\ldots,A_k$ such that
\[
\sum_{i=1}^k2^{\HH(A_i)}\leq s\times 2^{\HH(B)}.
\]
Moreover, if $s=1$ then $B$ can be chosen so that equality holds.
\end{lemma}
\begin{proof}
    We pick the independent random index $\bfi\in [k]$ as follows. For each $i\in[k]$, set
    \[p_i=\frac{2^{\HH(A_i)}}{\sum_{j=1}^k2^{\HH(A_j)}}\,\]
    and let $\PP(\bfi=i)=p_i$. Set $B=A_{\bfi}$. By the chain rule, we have $\HH(B)+\HH(\bfi\mid B) = \HH(\bfi) + \HH(B\mid \bfi) $. Note that
    \begin{align*}
        \HH(\bfi)+\HH(B\mid \bfi)&=\sum_{i=1}^kp_i\left(\HH(B\mid \bfi=i)-\log_2p_i\right)
        =\sum_{i=1}^kp_i\left(\HH(A_i)-\log_2p_i\right)
        \\&=\sum_{i=1}^kp_i\log_2\left(\sum_{j=1}^k2^{\HH(A_j)}\right)
        =\log_2\left(\sum_{j=1}^k2^{\HH(A_j)}\right).
    \end{align*}
    Also, by the definition of conditional entropy and the uniform bound, we have \[\HH(\bfi\mid B)=\sum_{b\in\supp(B)}P_B(b)\HH(\bfi\mid B=b)\leq \sum_{b\in\supp(B)}P_B(b)\log_2s=\log_2s,\] and the equality holds when $s=1$. Thus, 
    \[\log_2\left(\sum_{j=1}^k2^{\HH(A_j)}\right)=\HH(B)+\HH(\bfi\mid B)\leq \HH(B)+\log_2s,\]
    and the equality holds when $s=1$. This completes the proof.
\end{proof}

Incorporating the resampling trick and the mixture bound into the previous proof yields a stronger result which can be used to derive \cref{thm:rainbow-triangle}. In fact, this is the $r=3$ case of \cref{thm:entropy-r-color}. Say a distribution $(X_1,X_2,X_3)$ is \emph{symmetric} if $(X_1,X_2,X_3)$ and $(X_{\sg(1)},X_{\sg(2)},X_{\sg(3)})$ have the same marginal distribution for any permutation $\sg:[3]\to[3]$. 

\begin{theorem}[\cref{thm:entropy-r-color} for $r=3$]\label{thm:entropy_2rgb}
Let $G$ be a simple graph together with a 3-edge-coloring $\chi:E(G)\to \{\mathrm{red},\mathrm{green},\mathrm{blue}\}$. Let $(X_1,X_2,X_3)$ be any symmetric distribution on triples of vertices forming a rainbow triangle in $G$. Then $2\HH(X_1,X_2,X_3)\leq 3\HH(X_1,X_2)-\log_23$.
\end{theorem}

\begin{proof}
For $i\in[3]$, let $Y_i$ be a resample of $X_i$ given $X_{i-1},X_{i+1}$ and $\chi(X_{i-1}X_{i})$. Then let $Z_{i+1}$ be a resample of $X_{i+1}$ given $X_{i-1},Y_i$. As in the proof of \cref{thm:rainbow-triangle-simple}, all subscripts are taken modulo 3.

Mimicking the initial computation in the proof of \cref{thm:rainbow-triangle-simple}, we again derive
\[
\HH(Y_i,Z_{i+1}\mid X_1,X_2,X_3)=\HH(Y_i\mid X_1,X_2,X_3)+\HH(Z_{i+1}\mid X_1,X_2,X_3)
\geq 2\HH(X_3\mid X_1,X_2)-1
\]
for each $i\in[3]$.

\cref{claim:supports-id-ref-2} still holds in this setting, so we may apply the mixture bound (\cref{lem:mix}) to the six random tuples $(X_1,X_2,X_3,Y_i,Z_{i+1})$ for $i\in[3]$ and $(X_1,X_2,X_3,Z_{i+1},Y_i)$ for $i\in[3]$. This yields a mixture $(X_1,X_2,X_3,A,B)$ of these quintuples with the property that
\[
2^{\HH(X_1,X_2,X_3,A,B)}\geq \sum_{i\in[3]}\left(2^{\HH(X_1,X_2,X_3,Y_i,Z_{i+1})}+2^{\HH(X_1,X_2,X_3,Z_{i+1},Y_i)}\right)=2\times\sum_{i\in[3]}2^{\HH(X_1,X_2,X_3,Y_i,Z_{i+1})}.
\]
Dividing both sides by $2^{\HH(X_1,X_2,X_3)}$, dropping conditions on the left, and applying the bound we derived on the right yields
\[
2^{\HH(A,B)}\geq 2^{\HH(A,B\mid X_1,X_2,X_3)}\geq 2\times \sum_{i\in[3]} 2^{\HH(Y_i,Z_{i+1}\mid X_1,X_2,X_3)}\geq 6\times 2^{2\HH(X_3\mid X_1,X_2)-1}.
\]
Therefore,
\[
\HH(A,B)\geq 2\HH(X_3\mid X_1,X_2)+\log_23.
\]
To complete the proof, observe that $(Y_i,Z_{i+1})$ has the same marginal distribution as $(Y_i,X_{i+1})$, which in turn has the same marginal distribution as $(X_i,X_{i+1})$, because of how we sampled $Z_{i+1}$ and $Y_i$. Because $(X_1,X_2,X_3)$ has a symmetric distribution, this then is simply the distribution of $(X_1,X_2)$. Similarly, $(Z_{i+1},Y_i)$ has the same marginal distribution as $(X_{i+1},X_i)$, which is also the distribution of $(X_1,X_2)$. Thus, the mixture $(A,B)$ also has the same marginal distribution as $(X_1,X_2)$. It follows that
\begin{align*}
3\HH(X_1,X_2)&=2\HH(X_1,X_2)+\HH(A,B)
\\&\geq 2\HH(X_1,X_2)+2\HH(X_3\mid X_1,X_2)+\log_23=2\HH(X_1,X_2,X_3)+\log_23.\qedhere
\end{align*}
\end{proof}

Using \cref{thm:entropy_2rgb}, we can derive a new proof of \cref{thm:rainbow-triangle}.
\begin{proof}[Proof of \cref{thm:rainbow-triangle}]
Let $(X_1,X_2,X_3)$ be a uniformly random triple of vertices that form a rainbow triangle in $G$. Because each rainbow triangle contributes 6 ordered triples, we have $\HH(X_1,X_2,X_3)=\log_2(6T)$, recalling that $T$ is the number of rainbow triangles in $G$.

By \cref{thm:entropy_2rgb}, we know that $\HH(X_1,X_2,X_3)\leq \frac 32\HH(X_1,X_2)-\frac 12\log_23$. Moreover, the chain rule yields
\[\HH(X_1,X_2)=\HH(\chi(X_1X_2))+\HH(X_1,X_2\mid\chi(X_1X_2))\]
Because $(X_1,X_2,X_3)$ is equally likely to be any of the six permutations of a given rainbow triangle $(x_1,x_2,x_3)$, it follows that $\chi(X_1X_2)$ is uniformly distributed among the three colors $\{r,g,b\}$. Thus, 
\begin{align*}
    \HH(X_1,X_2)=&\HH(\chi(X_1X_2))+\frac 13\sum_{c\in\{r,g,b\}}\HH(X_1,X_2\mid\chi(X_1X_2)=c)\\
    \leq&\log_23+\frac 13(\log_2(2m_r)+\log_2(2m_g)+\log_2(2m_b)),
\end{align*}
where the final inequality follows from the uniform bound. Therefore,
\[
\log_2(6T)=\HH(X_1,X_2,X_3)\leq\frac 32\HH(X_1,X_2)-\frac 12\log_23\leq \log_23+\frac 12\log_2(8m_rm_gm_b).
\]
Rearranging yields $T\leq \sqrt{2m_rm_gm_b}$.
\end{proof}

\section{Proof of rainbow triangles with many colors}\label{sec:many_color_RT}
In this section, we give a proof of \cref{thm:r-color,thm:entropy-r-color}.
We first quickly show that \cref{thm:entropy-r-color} implies \cref{thm:r-color}.

\begin{proof}[Proof of \cref{thm:r-color} using \cref{thm:entropy-r-color}]
Let $(X_1,X_2,X_3)$ be a uniformly random rainbow triangle in $G$.
    Let $T$ be the number of rainbow triangles.
    Then by \cref{thm:entropy-r-color} and the uniform bound,
    \[\log_2 (6T) = \HH(X_1,X_2,X_3)\leq \frac{1}{2}\left(3\HH(X_1,X_2)+\log_2\frac{r-2}{r}\right)\leq \frac{1}{2}\log_2\left(\frac{r-2}{r}(2m)^3\right).\]
    Rearranging gives
    \[T\leq \sqrt{\frac{8(r-2)}{36r}}m^{3/2} = \sqrt{\frac{2(r-2)}{9r}}m^{3/2}.\qedhere\]
\end{proof}

In the remainder of the section, we prove \cref{thm:entropy-r-color}.

\begin{proof}[Proof of \cref{thm:entropy-r-color}]
Recall that $(X_1,X_2,X_3)$ is random rainbow triangle in $G$, chosen according to a symmetric but not necessarily uniform distribution.
Sample random variables as follows. 
This is depicted in \Cref{fig:rainbow-triangle-sampling}.
\begin{itemize}
\item Take $Y_1$ to be a conditionally independent resample of $X_1$ given $X_2,X_3,\chi(X_1X_2)$ and $\chi(X_1X_3)$.
\item Take $Y_2$ to be a conditionally independent resample of $X_2$ given $X_1,X_3$ and $\chi(X_1X_2)$.
\item Take $Y_3$ to be a conditionally independent resample of $X_3$ given $X_1$ and $X_2$.
\end{itemize}

\begin{figure}[h]	
\centering
\begin{tikzpicture}[trianglecolors,scale=0.8, every node/.append style={scale=0.8}]
\foreach \i in {1,2,3} {
	\path let \n1={int(\i+0)},\n2={int(mod(\i,3)+1)},\n3={int(mod(\i+1,3)+1)},
			\n{sqrt3}={sqrt(3)} in
	(120*\i-30:1) coordinate[vtx] (X\n1) 
		--+ (120*\i-30:0.4) node{$X_\n1$}
		--++ (120*\i-120:\n{sqrt3}) coordinate[openvtx] (Y\n2)
		--+ (120*\i-90:0.4) node{$Y_\n2$};
}

\begin{scope}[thick]
\draw[e1] (Y2) -- (X1) -- (X2) -- (Y1);
\draw[e2] (X2) -- (X3);
\draw[e3] (X1) -- (X3) -- (Y1) ;
\end{scope}

\begin{scope}[thick, dashed, gray]
\draw  (Y3) -- (X1) ;
\draw		(Y2) -- (X3) ;
\draw (X2) -- (Y3) ;
\end{scope}
\end{tikzpicture}
\caption{The possible colorings of the edges $X_{i\pm 1}Y_i$ given a rainbow triangle $(X_1,X_2,X_3)$ with the red--blue--green coloring shown. Dashed edges may receive any color, provided that the triangles are rainbow.}
\label{fig:rainbow-triangle-sampling}
\end{figure}
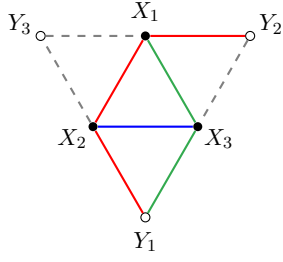

We begin by showing that the supports of $Y_1,Y_2,Y_3$ are disjoint given $X_1,X_2,X_3$.

\begin{claim}
The supports of $(X_1,X_2,X_3,Y_i)_{i\in [3]}$ are pairwise disjoint.
\end{claim}
\begin{proof}\proofofclaim
    Suppose that $(x_1,x_2,x_3,v)$ is in $\bigcup_{i\in [3]}\supp(X_1,X_2,X_3,Y_i)$.
    Then we see that $(x_1,x_2,x_3,v)$ is in the support of $\supp(X_1,X_2,X_3,Y_3)$ if and only if both $\chi(vx_1), \chi(vx_2)$ are both not equal to $\chi(x_1x_2)$.
    Moreover, if $(x_1,x_2,x_3,v)\not\in \supp(X_1,X_2,X_3,Y_3)$, then we can decide which of the remaining supports $(x_1,x_2,x_3,v)$ is in just from the color $\chi(vx_3)$: it is in $\supp(X_1,X_2,X_3,Y_1)$ only if $\chi(vx_3) = \chi(x_1x_3)$, and it is in $\supp(X_1,X_2,X_3,Y_2)$ only if $\chi(vx_3)\neq \chi(x_1x_3)$.
    This shows that the three supports are pairwise disjoint.
\end{proof}

Next we compute the conditional entropies of $Y_1,Y_2,Y_3$ given $X_1,X_2,X_3$.

\begin{claim}\label{claim:entropy-lower-bound}
We have the following (in)equalities on entropy.
\begin{align*}
\HH(Y_3\mid X_1,X_2,X_3)&=	\HH(X_3\mid X_1,X_2).
\\\HH(Y_2\mid X_1,X_2,X_3)&\geq \HH(X_3\mid X_1,X_2)-\log_2(r-1).
\\\HH(Y_1\mid X_1,X_2,X_3)&\geq \HH(X_3\mid X_1,X_2)-\log_2((r-1)(r-2)).
\end{align*}
\end{claim}

\begin{proof}\proofofclaim
Because $Y_3$ and $X_3$ are sampled according to the same distribution conditioned on $X_1,X_2$. It follows that $\HH(Y_3\mid X_1,X_2,X_3)=\HH(Y_3\mid X_1,X_2)=\HH(X_3\mid X_1,X_2)$.
Similarly, $Y_2$ and $X_2$ have the same distribution conditioned on $X_1$, $X_3$, and $\chi(X_1X_2)$. Thus,
\begin{align*}
\HH(Y_2\mid X_1,X_2,X_3)&=\HH(Y_2\mid X_1,X_3,\chi(X_1X_2))
=\HH(X_2\mid X_1,X_3,\chi(X_1X_2))
\\&= \HH(X_2, \chi(X_1X_2)\mid X_1,X_3)-\HH(\chi(X_1X_2)\mid X_1,X_3)
\\&\geq \HH(X_2\mid X_1,X_3)-\log_2(r-1)
\\&=\HH(X_3\mid X_1,X_2)-\log_2(r-1)\quad\text{by symmetry}.
\end{align*}
Similarly, $Y_1$ and $X_1$ have the same distribution conditioned on $X_2$, $X_3$, $\chi(X_1X_2)$, and $\chi(X_1X_3)$. Thus,
\begin{align*}
\HH(Y_1\mid X_1,X_2,X_3)&=\HH(Y_1\mid X_2,X_3,\chi(X_1X_2),\chi(X_1X_3))
=\HH(X_1\mid X_2,X_3,\chi(X_1X_2),\chi(X_1X_3))
\\&=\HH(X_1,\chi(X_1X_2),\chi(X_1X_3)\mid X_2,X_3)-\HH(\chi(X_1X_2),\chi(X_1X_3)\mid X_2,X_3)
\\&\geq \HH(X_1\mid X_2,X_3)-\log_2((r-1)(r-2))
\\&= \HH(X_3\mid X_1,X_2)-\log_2((r-1)(r-2))\quad\text{by symmetry}.\qedhere
\end{align*}
\end{proof}
To prove the theorem, we apply the mixture bound twice to get two inequalities (\cref{claim:many-color-1,claim:many-color-2}) that will imply the theorem when put together.

\begin{claim}\label{claim:many-color-1}
    $\HH(X_3\mid X_1,X_2)\leq \HH(X_1)+\log_2\frac{r-2}{r-1}$.
\end{claim}
\begin{proof}\proofofclaim
    By the mixture bound, there exists a mixture $(X_1,X_2,X_3,Y)$ of $(X_1,X_2,X_3,Y_i)_{i\in [3]}$ so that
    \[\sum_{i=1}^{3}2^{\HH(X_1,X_2,X_3,Y_i)}\leq 2^{\HH(X_1,X_2,X_3,Y)}.\]
    Since the marginal distribution of each $Y_i$ is the same as the distribution of $X_1$, we know that the marginal distribution of $Y$ is the same as the one of $X_1$ as well. Therefore, dividing both sides by $2^{\HH(X_1,X_2,X_3)}$ (via the chain rule) and dropping conditioning yields
    \[\sum_{i=1}^{3}2^{\HH(Y_i\mid X_1,X_2,X_3)}
    \leq 2^{\HH(Y\mid X_1,X_2,X_3)}\leq 2^{\HH(Y)}=2^{\HH(X_1)}.\]
    By \cref{claim:entropy-lower-bound}, we have
    \[\sum_{i=1}^{3}2^{\HH(Y_i\mid X_1,X_2,X_3)}\geq \left(1+\frac{1}{r-1}+\frac{1}{(r-1)(r-2)}\right)2^{\HH(X_3\mid X_1,X_2)} = \frac{r-1}{r-2}\cdot 2^{\HH(X_3\mid X_1,X_2)}.\]
    Rearranging gives us the desired inequality.
\end{proof}

\begin{claim}\label{claim:many-color-2}
    $\HH(X_3\mid X_1,X_2)\leq \HH(X_2\mid X_1)+\log_2\frac{r-1}{r}$.
\end{claim}
\begin{proof}\proofofclaim
    We run the same proof except, instead of mixing all $(X_1,X_2,X_3,Y_i)_{i\in [3]}$, we only mix the tuples for $i=2,3$.
    By the mixture bound, there exists a mixture $(X_1,X_2,X_3,Y')$ of $(X_1,X_2,X_3,Y_i)_{i=2,3}$ so that
    \[\sum_{i=2}^{3}2^{\HH(X_1,X_2,X_3,Y_i)}\leq 2^{\HH(X_1,X_2,X_3,Y')}.\]
    Note that the marginal distribution of $(X_1,Y_i)$ for $i=2,3$ is the same as the distribution of $(X_1,X_i)$, which by symmetry is the same as the distribution of $(X_1,X_2)$.
    Therefore the marginal distribution of $(X_1,Y')$ is the same as $(X_1,X_2)$ as well, and in particular $\HH(Y'\mid X_1)=\HH(X_2\mid X_1)$.
    Dividing both sides by $2^{\HH(X_1,X_2,X_3)}$ and dropping conditioning as in the proof of \cref{claim:many-color-1} now yields
    \[\sum_{i=2}^32^{\HH(Y_i\mid X_1,X_2,X_3)}
    \leq 2^{\HH(Y'\mid X_1,X_2,X_3)} 
    \leq 2^{\HH(Y'\mid X_1)} = 2^{\HH(X_2\mid X_1)}.\]
    On the other hand, by \cref{claim:entropy-lower-bound}, we have
    \[\sum_{i=2}^{3}2^{\HH(Y_i\mid X_1,X_2,X_3)}
    \geq\left(1+\frac 1{r-1}\right)2^{\HH(X_3\mid X_1,X_2)}
    = \frac{r}{r-1}\cdot 2^{\HH(X_3\mid X_1,X_2)}.\]
    Rearranging gives the desired inequality.
\end{proof}

To complete the proof, simply sum up the inequalities in \cref{claim:many-color-1,claim:many-color-2} to get
\[2\HH(X_3\mid X_1,X_2)\leq \HH(X_1,X_2)+\log_2\left(\frac{r-2}{r}\right).\]
Adding $2\HH(X_1,X_2)$ to both sides and applying the chain rule completes the proof.
\end{proof}

\section{Properly colored \texorpdfstring{$K_d$}{Kd}'s}\label{sec:proper_Kd}

We now count properly colored $K_d$'s for $d\geq 4$. The goal of this section is to give a proof of \cref{thm:properly-colored-Kd-in-r-colors} by proving the following entropic version of it.

\begin{theorem}\label{thm:entropy-properly-colored-Kd-in-r-colors}
    Let $d,r$ be two positive integers with $d\geq 4$ and $r\geq d-1$.
    Let $G$ be a graph whose edges are colored with $r$ colors.
    Let $(X_1,\ldots, X_d)$ be a symmetric distribution supported on properly colored $K_d$'s in $G$.
    Then
    \[2\HH(X_1,\ldots, X_d)\leq d\HH(X_1,X_2)+\log_2\left(\frac{(r-1)^2\cdots (r-d+2)^2}{(r+1)^{d-2}r^{d-2}}\right).\]
\end{theorem}

\begin{proof}[Proof of \cref{thm:properly-colored-Kd-in-r-colors} using \cref{thm:entropy-properly-colored-Kd-in-r-colors}]
    Let $(X_1,\ldots, X_d)$ be a uniformly random properly colored $K_d$ in $G$ with a uniformly chosen ordering.
    Let $N_d$ be the number of properly colored $K_d$'s in $G$.
    Then by \cref{thm:entropy-properly-colored-Kd-in-r-colors} and the uniform bound,
    \[\log_2 d!N_d = \HH(X_1,\ldots, X_d)\leq \frac{d}{2}\log_2(2m)+\frac{1}{2}\log_2\left(\frac{(r-1)^2\cdots (r-d+2)^2}{(r+1)^{d-2}r^{d-2}}\right).\]
    Rearranging gives
    \[N_d\leq \frac{(r-1)\dots (r-d+2)}{d!}\sqrt{\frac{(2m)^d}{(r+1)^{d-2}r^{d-2}}} = \binom{r+1}{d}\left(\frac{2m}{(r+1)r}\right)^{d/2} = \binom{r+1}{d}\left(\frac{m}{\binom{r+1}{2}}\right)^{d/2}.\]
\end{proof}

The remaining of this section will be dedicated to proving \cref{thm:entropy-properly-colored-Kd-in-r-colors}.
When reading the proofs, it is helpful to keep in mind that \cref{thm:properly-colored-Kd-in-r-colors,thm:entropy-properly-colored-Kd-in-r-colors}, as well as all intermediate statements in this section, are tight when the number of colors $r$ is odd and $G$ is a blowup of a properly edge-colored $K_{r+1}$ .

\begin{proof}[Proof of \cref{thm:entropy-properly-colored-Kd-in-r-colors}]
    Our goal is to prove the following two claims.
    \begin{claim}\label{claim:Kd-first-ineq}
        For every $3\leq k\leq d$,
        \[\HH(X_k\mid X_1,\ldots, X_{k-1})\leq \HH(X_2\mid X_1)+\log_2\left(\frac{r-k+2}{r}\right).\]
    \end{claim}
    \begin{claim}\label{claim:Kd-second-ineq}
        For every $4\leq k\leq d$,
        \[\HH(X_k\mid X_1,\ldots, X_{k-1})\leq \HH(X_1)+\log_2\left(\frac{r-k+2}{r+1}\right).\]
    \end{claim}
    Assuming we have proven both claims, then we get that 
    \[2\HH(X_k\mid X_1,\ldots, X_{k-1})\leq \HH(X_1,X_2)+\log_2\left(\frac{(r-k+2)^2}{r(r+1)}\right)\]
    for every $k\geq 4$, though we will only use this for $k\geq 5$.
    By the chain rule, this shows that
    \[2\HH(X_5,\ldots, X_d\mid X_1,X_2,X_3,X_4)\leq (d-4)\HH(X_1,X_2)+\log_2\left(\frac{(r-3)^2\cdots (r-d+2)^2}{r^{d-4}(r+1)^{d-4}}\right).\]
    Lastly, we know that
    \begin{align*}
        2\HH(X_1,X_2,X_3,X_4) =& 2\HH(X_1,X_2)+2\HH(X_3\mid X_1,X_2)+2\HH(X_4\mid X_1,X_2,X_3)\\
        \leq& 2\HH(X_1,X_2)+2\left[\HH(X_2\mid X_1)+\log_2\left(\frac{r-1}{r}\right)\right]+2\left[\HH(X_1)+\log_2\left(\frac{r-2}{r+1}\right)\right]\\
        =&4\HH(X_1,X_2)+\log_2\left(\frac{(r-1)^2(r-2)^2}{r^2(r+1)^2}\right).
    \end{align*}
    Adding the two inequalities up gives the desired upper bound on $2\HH(X_1,\ldots,X_d)$.

    We will now prove the two upper bounds on $\HH(X_k\mid X_1,\ldots, X_{k-1})$.
    As such, we fix some $k\in [3,d]$.
    For any $i\in[k]$, let $Y_i$ be a resample of $X_i$ given all the other $k-1$ random variables among $(X_1,\ldots, X_k)$.
    For any $i,j\in[k]$ with $i\neq j$, let $Z_{ij}$ be a resample of $X_i$ given all the other $k-1$ random variables $(X_1,\ldots, X_k)$ and also $\chi(X_iX_j)$.

\begin{figure}[h]	
\centering
\begin{tikzpicture}[trianglecolors,scale=1.3, every node/.append style={scale=1.0}]
\foreach \i in {1,2,3,4} {
	\path (90*\i-45:1) coordinate[vtx] (X\i) 
		--+ (90*\i-45:0.3) node{$X_\i$};
}

\path (20:1) coordinate[openvtx] (Y1) --+ (45:0.3) node{$Y_1$};
\path (160:1) coordinate[openvtx] (Z23) --+ (135:0.3) node{$Z_{23}$};

\begin{scope}[thick]
\draw[e1] (X1) -- (X2);
\draw[e2] (X2) -- (X3) -- (Z23);
\draw[e3] (X1) -- (X3);
\draw[e4] (X4) -- (X1);
\draw[e5] (X4) -- (X2);
\draw[e6] (X4) -- (X3);
\end{scope}

\begin{scope}[thick, dashed, gray]
\draw  (X2) -- (Y1) -- (X3);
\draw (X4) -- (Y1);
\draw (X1) -- (Z23) -- (X4);
\end{scope}
\end{tikzpicture}
\caption{A diagram for the sampled random variables $Y_1, Z_{23}$ when $k=4$. Dashed edges may receive any color, provided that all the $K_4$'s are properly colored.}
\label{fig:k4-sampling}
\end{figure}
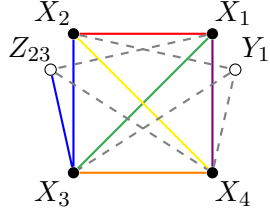

    \begin{claim}\label{claim:Kd-entropy-lower-bound}
        For any $i,j\in[k]$ with $i\neq j$, we have
        \[\HH(Y_i\mid X_1,\ldots, X_k) = \HH(X_k\mid X_1,\ldots, X_{k-1})\]
        and
        \[\HH(Z_{ij}\mid X_1,\ldots, X_k)\geq \HH(X_k\mid X_1,\ldots, X_{k-1})-\log_2(r-k+2).\]
    \end{claim}
    \begin{proof}\proofofclaim
        As $Y_i$ is a resample of $X_i$ given the other $k-1$ random variables, we have
        \[\HH(Y_i\mid X_1,\ldots, X_k)=\HH(X_i\mid X_1,\ldots, X_{i-1},X_{i+1},\ldots, X_k) = \HH(X_k\mid X_1,\ldots, X_{k-1})\]
        where the last inequality comes from symmetry.
        Similarly, we have
        \[\HH(Z_{ij}\mid X_1,\ldots, X_k) = \HH(X_k\mid X_1,\ldots, X_{k-1},\chi(X_1X_k)).\]
        By the chain rule, we can expand
        \begin{align*}
            \HH(Z_{ij}\mid X_1,\ldots, X_k)
            &=\HH(X_k, \chi(X_1X_k)\mid X_1,\ldots, X_{k-1})-\HH(\chi(X_1X_k)\mid X_1,\ldots, X_{k-1})\\
            &\geq\HH(X_k\mid X_1,\ldots, X_{k-1})-\log_2(r-k+2)
        \end{align*}
        where we use the uniform bound to bound the latter term.
    \end{proof}

    As before, we would like to show that the supports of the $Y_i$'s and the $Z_{ij}$'s are disjoint given $X_1,\ldots, X_k$.
    This time they are not pairwise disjoint as clearly the support of $Z_{ij}$ is contained in the support of $Y_i$.
    However, we can still show that some supports are disjoint.

    \begin{claim}\label{claim:weak-support-disjoint}
        We have that $\supp(X_1,\ldots, X_k, Z_{ij})$ is disjoint from $\supp(X_1,\ldots, X_k, Y_{i'})$ for any three pairwise distinct indices $i,j,i'\in[k]$. 
    \end{claim}
    \begin{proof}\proofofclaim
        Suppose that $(x_1,\ldots, x_k,v)\in \supp(X_1,\ldots, X_k, Z_{ij})\cap \supp(X_1,\ldots, X_k, Y_{i'})$.
        Then as $v$ can play the role of $Z_{ij}$, we see that $\chi(vx_j) = \chi(x_ix_j)$.
        However, as $v$ can play the role of $Y_{i'}$ as well, we know that $x_1,\ldots, x_{i'-1},v,x_{i'+1},\ldots, x_k$ is a properly colored $K_k$.
        In particular, we have $\chi(vx_j)\neq \chi(x_ix_j)$, which is a contradiction.
    \end{proof}
    Now we can prove \cref{claim:Kd-first-ineq,claim:Kd-second-ineq}.

    \begin{proof}[Proof of \cref{claim:Kd-first-ineq}]\proofofclaim
        Because $\supp(X_1,\ldots, X_k,Z_{i1})\subseteq \supp(X_1,\ldots, X_k,Y_i)$ for any $i\neq 1$, applying \cref{claim:weak-support-disjoint} shows that $(X_1,\ldots, X_k,Y_2)$ and $(X_1,\ldots, X_k,Z_{i1})_{3\leq i\leq k}$ have pairwise disjoint supports.
        Therefore by the mixture bound (\cref{lem:mix}), there is a mixture $(X_1,\ldots, X_k, W)$ of these $k-1$ random tuples so that
        \[2^{\HH(X_1,\ldots, X_k,W)}\geq 2^{\HH(X_1,\ldots, X_k,Y_2)}+\sum_{i=3}^{k}2^{\HH(X_1,\ldots, X_k,Z_{i1})}.\]
        Dividing both sides by $2^{\HH(X_1,\ldots,X_k)}$ (via the chain rule) and applying \cref{claim:Kd-entropy-lower-bound} yields
        \[
        2^{\HH(W\mid X_1,\ldots,X_k)}\geq 2^{\HH(Y_2\mid X_1,\ldots,X_k)}+\sum_{i=3}^k2^{\HH(Z_{i1}\mid X_1,\ldots,X_k)}\geq 2^{\HH(X_k\mid X_1,\ldots, X_{k-1})}\left(1+\frac{k-2}{r-k+2}\right),
        \]
        or equivalently
        \[\HH(W\mid X_1,\ldots, X_k)\geq \HH(X_k\mid X_1,\ldots, X_{k-1})+\log_2\left(\frac{r}{r-k+2}\right).\]
        Now notice that the marginal distributions of $(X_1,Y_2)$ and $(X_1,Z_{i1})_{3\leq i\leq k}$ are each identically to the distribution of $(X_1,X_2)$.
        Therefore $(X_1,W)$ is also distributed as $(X_1,X_2)$, showing that
        \[\HH(W\mid X_1,\ldots,X_k)\leq \HH(W\mid X_1) = \HH(X_2\mid X_1).\]
        Plugging this back in, we get
        \[\HH(X_k\mid X_1,\ldots, X_{k-1})\leq \HH(X_2\mid X_1)+\log_2\left(\frac{r-k+2}{r}\right),\]
        which is the desired inequality.
    \end{proof}

    \begin{proof}[Proof of \cref{claim:Kd-second-ineq}]\proofofclaim
        We begin by showing that the supports of the random tuples $(X_1,\ldots, X_k,Y_i)_{i\in [k]}$ and $(X_1,\ldots, X_k,Z_{ij})_{i\neq j\in [k]}$ are $(k+1)$-wise disjoint.
        To see this, fix any $(x_1,\ldots, x_k,v)$ in the union of the supports.
        Let $I$ be the subset of $[k]$ consisting of indices $i$ for which $ \supp(X_1,\ldots, X_k, Y_i)$ contains $(x_1,\ldots, x_k,v)$.
        For any $i\in I$, let $J_i$ be the subset of $[k]\backslash \{i\}$ consisting of indices $j$ with $(x_1,\ldots, x_k,v)\in \supp(X_1,\ldots, X_k,Z_{ij})$.
        Our aim is to prove that
        \[\abs{I}+\sum_{i\in I}\abs{J_i}\leq k.\]

        Recall that \cref{claim:weak-support-disjoint} says that it is impossible to find $i,i',j\in[k]$ pairwise distinct so that $i,i'\in I$ and $j\in J_i$.
        If $\abs{I}=1$, then the inequality holds trivially as $\abs{J_i}\leq k-1$ for each $i\in I$.
        If $\abs{I}\geq 3$, then we have that $\abs{J_i}=0$ for each $i\in I$: otherwise, for any $j\in J_i$ we can always find $i'\in I\backslash \{i,j\}$, which contradicts \cref{claim:weak-support-disjoint}.
        Thus the inequality also holds as $\abs{I}\leq k$.

        The remaining case is $\abs{I}=2$, where we will finally use the assumption $k\geq 4$.
        Suppose that $I = \{i,i'\}$.
        Then \cref{claim:weak-support-disjoint} implies that $J_i\subseteq \{i'\}$ and $J_{i'}\subseteq \{i\}$, showing that
        \[\abs{I}+\sum_{i\in I}\abs{J_i}\leq 2+(1+1) = 4\leq k,\]
        as desired.

        Now we can apply the mixture bound (\cref{lem:mix}) in its most general form to get a mixture $(X_1,\ldots, W')$ of the tuples $(X_1,\ldots, X_k,Y_i)_{i\in [k]}$ and $(X_1,\ldots, X_k, Z_{ij})_{i\neq j\in [k]}$ that satisfies
        \[k\cdot 2^{\HH(X_1,\ldots, X_k,W')}\geq \sum_{i\in [k]}2^{\HH(X_1,\ldots X_k, Y_i)}+\sum_{i\neq j\in[k]}2^{\HH(X_1,\ldots, X_k,Z_{ij})}.\]
        Dividing both sides by $2^{\HH(X_1,\ldots,X_k)}$ (via the chain rule) and applying \cref{claim:Kd-entropy-lower-bound} yields
        \[k\cdot 2^{\HH(W'\mid X_1,\ldots, X_k)}\geq 2^{\HH(X_k\mid X_1,\ldots, X_{k-1})}\left(k+\frac{k(k-1)}{r-k+2}\right).\]
        We may cancel out the $k$ and take logarithms to get
        \[\HH(W'\mid X_1,\ldots, X_k)\geq \HH(X_k\mid X_1,\ldots, X_{k-1})+\log_2\left(\frac{r+1}{r-k+2}\right).\]
        On the other hand, we know that the marginal distribution of $Y_i$ for each $i\in [k]$ and the marginal distribution of $Z_{ij}$ for each $i\neq j\in [k]$ are the same as the distribution of $X_1$.
        Therefore $W'$ is identically distributed as $X_1$ as well, showing that
        \[\HH(W'\mid X_1,\ldots, X_k)\leq \HH(W') = \HH(X_1).\]
        Combining the two inequalities gives the desired upper bound on $\HH(X_k\mid X_1,\ldots, X_{k-1})$.
    \end{proof}
    Since the two claims are proven, the theorem statement follows.
\end{proof}

\section{Structural theorem and stability}\label{sec:sturctural_stability}
In this section, except for \cref{subsec:total-variation}, we will use the following notation.
Let $G$ be a simple $r$-edge-colored graph, and let $(X_1,X_2,X_3)$ be a random rainbow triangle in $G$ following a symmetric (but not necessarily uniform) distribution.
Recall in the proof of \cref{thm:entropy-r-color} we sampled the following.
\begin{itemize}
  \setlength\itemsep{0.1em}
    \item $Y_1$ is a conditionally independent resample of $X_1$ given $X_2,X_3,\chi(X_1X_2),\chi(X_1X_3)$.
    \item $Y_2$ is a conditionally independent resample of $X_2$ given $X_1,X_3,\chi(X_1X_2)$.
    \item $Y_3$ is a conditionally independent resample of $X_3$ given $X_1,X_2$.
    \item $Y$ is a mixture of $Y_1,Y_2,Y_3$ such that 
    $2^{\HH(Y\mid X_1,X_2,X_3)} = \sum_{i=1}^{3}2^{\HH(Y_i\mid X_1,X_2,X_3)}$.
    \item $Y'$ is a mixture of $Y_2,Y_3$ such that $2^{\HH(Y'\mid X_1,X_2,X_3)} = \sum_{i=2}^{3}2^{\HH(Y_i\mid X_1,X_2,X_3)}$.
\end{itemize}
Here, we use the special case of \cref{lem:mix} that the equality holds since the supports are pairwise disjoint. Let 
\[p_i\eqdef\frac{2^{\HH(Y_i,X_1,X_2,X_3)}}{\sum_{j=1}^32^{\HH(Y_j,X_1,X_2,X_3)}}\,\forall i\in [3].\]
Recall from \cref{lem:mix} that $Y$ is the mixture of $Y_1,Y_2,Y_3$ with probabilities $(p_1,p_2,p_3)$ and $Y'$ is the mixture of $Y_2,Y_3$ with probabilities $(\frac{p_2}{p_2+p_3},\frac{p_3}{p_2+p_3})$.
Finally, let 
\[\veps \eqdef 3\HH(X_1,X_2)+\log_2\frac{r-2}{r}-2\HH(X_1,X_2,X_3)\geq 0\]
measure the amount of slack in \cref{thm:entropy-r-color}.

\subsection{Structural theorem}\label{subsec:structural}

In this subsection, we will give a proof of \cref{thm:tight-construction}.
We will thus take $G$ to be a simple $r$-edge-colored graph with $m$ edges and $\sqrt{\frac{2(r-2)}{9r}}m^{3/2}$ rainbow triangles.
Let $(X_1,X_2,X_3)$ be a uniformly random rainbow triangle in $G$.
Then we can compute
\[3\HH(X_1,X_2)+\log_2\frac{r-2}{r} \leq \log_2 \frac{8(r-2)m^3}{r}\]
and
\[2\HH(X_1,X_2,X_3) = 2\log_2\left(6\times \sqrt{\frac{2(r-2)}{9r}}m^{3/2}\right) = \log_2\frac{8(r-2)m^3}{r}.\]
This forces the first equality to hold by \cref{thm:entropy-r-color}, and we have $\varepsilon=0$.
By examining the proof of \cref{thm:r-color}, we can extract the following information as stated in \cref{lem:structure-easy-corollary}.
As the proof relies on simple but somewhat long entropic computations, we defer those computations to \cref{subsec:entropy_input}.

\begin{lemma}\label{lem:structure-easy-corollary}
The following are all true when $G$ has exactly $\sqrt{\frac{2(r-2)}{9r}}m^{3/2}$ rainbow triangles and $(X_1,X_2,X_3)$ is a uniformly random rainbow triangle in $G$.
\begin{itemize}
  \setlength\itemsep{0.1em}
    \item  We have $p_1 = \frac{1}{(r-1)^2}$, $p_2 = \frac{r-2}{(r-1)^2}$ and $p_3 = \frac{r-2}{r-1}$.
    \item $Y$ is independent of $X_1,X_2,X_3$.
    \item $Y'$ is conditionally independent of $X_2,X_3$ given $X_1$.
    \item Given $X_1,X_2,X_3$, the color $\chi(X_3Y_2)$ distributes uniformly among $[r]\backslash \{\chi(X_1X_2),\chi(X_1X_3)\}$.
    \item Given $X_1,X_2,X_3$, the color pair $(\chi(X_1Y_3),\chi(X_2Y_3))$ distributes uniformly among the color pairs $(c_1,c_2)$ with $c_1,c_2\in [r]\backslash \{\chi(X_1X_2)\}$ and $c_1\neq c_2$.
\end{itemize}
\end{lemma}

As $Y$ is a mixture of $Y_1$ and $Y'$, \cref{lem:structure-easy-corollary} immediately implies that $Y_1$ conditionally independent of $X_2,X_3$ given $X_1$.
This is going to be the key ingredient in the proof.
In fact, we can use this observation and \cref{lem:structure-easy-corollary} to show the following lemma.

To state our next lemma, we need some definitions.
For any rainbow triangle $x=(x_1,x_2,x_3)$, let $V_1^{(x)}$ be the support of $Y_1$ given $(X_1,X_2,X_3)=(x_1,x_2,x_3)$.
Let $V_2^{(x)}\cup\cdots \cup V_{r-1}^{(x)}$ be the support of $Y_2$ given $(X_1,X_2,X_3)=(x_1,x_2,x_3)$, and partition this support into $V_2^{(x)},\ldots, V_{r-1}^{(x)}$ based on the color $\chi(x_3Y_2)$, which lies in the set $[r]\backslash \{\chi(x_1x_2), \chi(x_1x_3)\}$ of size $r-2$.
Lastly, let $V^{(x)}_r\cup\cdots\cup V^{(x)}_{(r-1)^2}$ be the support of $Y_3$ given $(X_1,X_2,X_3)=(x_1,x_2,x_3)$, and partition this support into $V^{(x)}_r,\ldots, V^{(x)}_{(r-1)^2}$ based on the color pair $(\chi(x_1Y_3),\chi(x_2Y_3))$, where it ranges over $(c_1,c_2)\in [r]^2$ with $c_1,c_2,\chi(x_1x_2)$ all distinct.
Recall that $V_1^{(x)},\ldots, V_{(r-1)^2}^{(x)}$ are disjoint sets of vertices. For intuition, note that if $G$ is a blowup of a graph induced by an affine plane as pictured in \Cref{fig:affine-plane-coloring}, then $V_1^{(x)},\ldots, V_{(r-1)^2}^{(x)}$ are the $(r-1)^2$ parts of the blowup in some order.

\begin{lemma}\label{lem:structure-X1-Y1}
Suppose $G$ has exactly $\sqrt{\frac{2(r-2)}{9r}}m^{3/2}$ rainbow triangles and $(X_1,X_2,X_3)$ is uniformly distributed over rainbow triangles.
    For any rainbow triangle $x=(x_1,x_2,x_3)$ in $G$, the marginal distribution of $(X_1,Y_1)$ is the same as the following distribution: first sample $\bfi\in [(r-1)^2]$ uniformly, and then sample two independent uniform vertices from $V_{\bfi}$.
    In particular, the family of sets $\left\{V_1^{(x)},\ldots, V_{(r-1)^2}^{(x)}\right\}$ does not depend on the ordered triple $(x_1,x_2,x_3)$.
\end{lemma}
\begin{proof}
    Let $X = (X_1,X_2,X_3)$ and consider the random sets $V_1^{(X)},\ldots, V_{(r-1)^2}^{(X)}$ constructed above.
    Recall also that $\PP(Y=Y_i) = p_i$ for all $i\in[3]$.
    Conditioning on $Y=Y_3$, we know that it is in each of $V_r^{(X)},\ldots, V_{(r-1)^2}^{(X)}$ with equal probability by \cref{lem:structure-easy-corollary}.
    Therefore $\PP(Y\in V_j^{(X)}) = p_3/(r-1)(r-2) = (r-1)^{-2}$ for any $r\leq j\leq (r-1)^2$.
    A similar argument shows that $Y$ lies in each of $V_1^{(X)},\ldots, V_{(r-1)^2}^{(X)}$ with probability $p_2/(r-2)=(r-1)^{-2}=p_1$.

    Now let $\mathbf{t}\in[3]$ be such that $Y=Y_{\mathbf{t}}$, and let $\bfi\in[(r-1)^2]$ be such that $Y\in V_{\bfi}^{(X)}$.
    From the discussion above, we see $\bfi$ is uniform and independent of $(X_1,X_2,X_3)$.
    Now sample $Y^*$ conditionally independently given $X, \bfi$ and uniformly at random from $V_{\bfi}^{(X)}$.
    The key observation is that $(Y, X_{\mathbf{t}+1}, X_{\mathbf{t}+2}, Y^*)$ is identically distributed as $(X_1,X_2,X_3,Y_1)$.
    Recall that $(Y_1,X_2,X_3)$, $(Y_2,X_3,X_1)$, and $(Y_3,X_1,X_2)$ are all identically distributed as $(X_1,X_2,X_3)$.
    Thus their mixture $(Y,X_{\mathbf{t}+1},X_{\mathbf{t}+2})$ is also distributed as $(X_1,X_2,X_3)$.
    From the definition of $V_{\bfi}^{(X)}$, we know that $Y^*$ is uniformly chosen from the vertices that form, with $X_{\mathbf{t}+1},X_{\mathbf{t}+2}$, a rainbow triangle with the same coloring as the rainbow triangle $YX_{\mathbf{t}+1}X_{\mathbf{t}+2}$.
    From the definition of $Y_1$, we see that $(Y, X_{\mathbf{t}+1}, X_{\mathbf{t}+2}, Y^*)$ is identically distributed as $(X_1,X_2,X_3,Y_1)$.
    As $(X_2,X_3)$ and $Y_1$ are conditionally independent given $X_1$, we see that $(X_{\mathbf{t}+1},X_{\mathbf{t}+2})$ and $Y^*$ are conditionally independent given $Y$.
    This shows that the set $V_{\bfi}^{(X)}$ only depends on $Y$ and does not depend on $X=(X_1,X_2,X_3)$ at all.
    As a consequence, the family of sets $\left\{V_1^{(X)},\ldots, V_{(r-1)^2}^{(X)}\right\}$ does not depend on $X$ at all, showing that it is equal to $\left\{V_1^{(x)},\ldots, V_{(r-1)^2}^{(x)}\right\}$ as well.
    Moreover, as $(X_1,Y_1)$ is symmetric, we know that $(Y,Y^*)$ is symmetric as well.
    Since $Y$ is equally likely to lie in each of $V_1^{(x)},\ldots, V_{(r-1)^2}^{(x)}$ and $Y^*$ is chosen uniformly from that set, we see that $(Y,Y^*)$ must follow the distribution as described in the lemma statement.
    We are then done as $(Y,Y^*)$ is identically distributed as $(X_1,Y_1)$.
\end{proof}

We will now fix the partition $V_1^{(x)},\ldots, V_{(r-1)^2}^{(x)}$ given by \cref{lem:structure-X1-Y1} according to an arbitrary rainbow triangle $x=(x_1,x_2,x_3)$.
We omit the superscript $x$ if it is clear from context.
\cref{lem:structure-X1-Y1} tells us what the distribution of $(X_1,Y_1)$ is.
Note that we also know the distribution of $(X_1,Y)$ as $Y$ is independent of $X_1$ and also identically distributed as $X_1$.
Using this, it is easy to deduce the distribution of $(X_1,Y')$, which is the same as the distribution of $(X_1,X_2)$.
This will allow us to show that $G$ is a balanced blowup of some complete colored graph on $(r-1)^2$ vertices.

\begin{lemma}\label{lem:structure-X1-Yprime}
    Suppose $G$ has exactly $\sqrt{\frac{2(r-2)}{9r}}m^{3/2}$ rainbow triangles and $(X_1,X_2,X_3)$ is uniformly distributed over rainbow triangles.
    The marginal distribution of $(X_1,X_2)$ is the same as the following distribution: first sample $\bfi,\bfj\in [(r-1)^2]$ uniformly subject to $\bfi \neq \bfj$ and then sample one vertex uniformly from each of $V_{\bfi}, V_{\bfj}$.
    In addition, $\abs{V_1}=\cdots =\abs{V_{(r-1)^2}}$ and for any $i\neq j\in [(r-1)^2]$, $G$ is monochromatic and complete between $V_i$ and $V_j$.
\end{lemma}
\begin{proof}
    By \cref{lem:structure-X1-Y1}, $(X_1,Y_1,Y)$ can be sampled as follows: first sample $\bfi, \bfi^*\in [(r-1)^2]$ independently and uniformly, next sample $X_1,Y_1$ uniformly from $V_{\bfi}$, and lastly sample $Y$ uniformly from $V_{\bfi^*}$ unless $\bfi=\bfi^*$, in which case set $Y=Y_1$.
    Since $Y$ is a mixture of $Y_1$ and $Y'$ with probabilities $(\frac{1}{(r-1)^2}, \frac{r(r-2)}{(r-1)^2})$, this means that $Y'$ should be sampled uniformly from $V_{\bfj}$ where $\bfj$ is sampled uniformly from $[(r-1)^2]\backslash \{\bfi\}$.
    As $(X_1,Y')$ is identically distributed as $(X_1,X_2)$, the first claim follows.

    Since we are in the extremal case, we know that $(X_1,X_2)$ is a uniform oriented edge in $G$.
    This shows that $\abs{V_1}=\cdots=\abs{V_{(r-1)^2}}$.

    Lastly, by the definition of $Y_1$, we know that $\chi(X_1X_2) = \chi(Y_1X_2)$.
    On the other hand, we know that $Y_1$ is conditionally independent of $X_2$ given $X_1$, showing that we can sample $(X_1,X_2,Y_1)$ just like how we sampled $(X_1,Y',Y_1)$ above.
    In particular, for any $i\neq j\in [(r-1)^2]$, we know that $G$ is complete between $V_i,V_j$ and also for any $x_1,y_1\in V_i$ and $x_2\in V_j$ we have $\chi(x_1x_2) = \chi(y_1x_2)$.
    Swapping the role of $i,j$, we see that $G$ is also monochromatic between $V_i$ and $V_j$.
\end{proof}

Let $H$ be the auxiliary complete graph on $[(r-1)^2]$ where $\chi(ij)$ is the color of the biclique between $V_i$ and $V_j$ in $G$.
Then $G$ is a balanced blowup of $H$ with possibly some isolated vertices, and it suffices to show that $H$ comes from some affine plane of order $(r-1)$.

Recall that the family $\left\{V_1^{(x)}, \ldots, V_{(r-1)^2}^{(x)}\right\}$ does not depend on $x_1x_2x_3$.
What this means for $H$ is that for any rainbow triangle $ijk$ in $H$, there is exactly one index $i'$ with $\chi(i'j)=\chi(ij), \chi(i'k)=\chi(ik)$ (which is $i'=i$), exactly $(r-2)$ indices $j'$ with $\chi(j'k) = \chi(jk)$ and $ij'k$ rainbow, and exactly $(r-2)(r-1)$ indices $k'$ with $ijk'$ rainbow. 
Since each edge is in some rainbow triangle by \cref{lem:structure-X1-Yprime}, we can end with the following lemma.

\begin{lemma}\label{lem:extract-affine-plane}
    Let $H$ be an $r$-edge-colored complete graph on $[(r-1)^2]$.
    Suppose that for each $i\neq j$ in $[(r-1)^2]$, there are exactly $(r-2)$ indices $j'$ with $\chi(ij') = \chi(ij)$ and exactly $(r-1)(r-2)$ indices $k$ with $ijk$ forming a rainbow triangle.
    Then $H$ is a graph induced by an affine plane of order $(r-1)$.
\end{lemma}
\begin{proof}
    Note that if $\chi(ik)=\chi(ij)$ then $ijk$ is not a rainbow triangle.
    Therefore for any $k\neq i,j$, we know that either $ijk$ is a rainbow triangle or $\chi(ik)=\chi(ij)$.
    By symmetry we know that if $ijk$ is not a rainbow triangle then $\chi(jk)=\chi(ij)$ as well.
    This shows that for any distinct $i,j,k\in [(r-1)^2]$, if $\chi(ij)=\chi(jk)$, then $ijk$ is a monochromatic triangle.
    As an immediate consequence, each color class must be a vertex-disjoint union of some number of cliques.
    The assumption shows that each clique must be of size $(r-1)$, and each vertex is in such a clique of any given color.
    In other words, each color class must be a union of $(r-1)$ vertex-disjoint $(r-1)$-cliques, as desired.
\end{proof}

Since the auxiliary $H$ satisfies the assumptions in \cref{lem:extract-affine-plane} and $G$ is a balanced blowup of $H$ with potentially some isolated vertices, \cref{thm:tight-construction} is now proven assuming \cref{lem:structure-easy-corollary}.

\subsection{Entropy input}\label{subsec:entropy_input}
In this subsection, we carry out the entropic computations for both the proof of \cref{thm:tight-construction} and the proof of \cref{thm:stability_brief}.
We briefly remark that for the proof of \cref{thm:tight-construction}, we just need (a) throguh (e) of \cref{lem:entropy_input} to prove \cref{lem:structure-easy-corollary}, whereas we will also need (f) and (g) for \cref{thm:stability_brief}.
\begin{lemma}\label{lem:entropy_input}
    The following are all true.
    \begin{enumerate}[label=\textup{(\alph*)}]
        \item $\HH(Y\mid X_1,X_2,X_3)\geq \HH(X_1)-\veps$.
        \item $\HH(Y'\mid X_1,X_2,X_3)\geq \HH(X_2\mid X_1)-\veps$.
        \item $p_1=\frac{1}{(r-1)^2}+O(\veps)$, $p_2=\frac{r-2}{(r-1)^2}+O(\veps)$, $p_3=\frac{r-2}{r-1}+O(\veps)$.
        \item $\HH(\chi(X_3Y_2)\mid X_1,X_2,X_3)\geq \log_2(r-2)-O(r^2\veps)$.
        \item $\HH(\chi(X_1Y_3),\chi(X_2Y_3)\mid X_1,X_2,X_3)\geq \log_2((r-1)(r-2))-O(r^2\veps)$.
        \item $\HH(\chi(X_1X_2)\mid X_1)\geq \log_2r-O(r\veps)$.
        \item $\HH(Y_2\mid X_1,X_2,X_3)\geq \HH(X_2\mid X_1, \chi(X_1X_2))-O(r\varepsilon)$. 
    \end{enumerate}
    
    \begin{proof}
    We begin with some preparation.
    By the definition, we have $2\HH(X_3\mid X_1,X_2)= \HH(X_1,X_2)+\log_2\frac{r-2}{r}-\varepsilon$.
    Thus, by \cref{claim:many-color-1} and \cref{claim:many-color-2}, we see that we have \[\HH(X_3\mid X_1,X_2)\geq \HH(X_1)+\log_2\frac{r-2}{r-1}-\varepsilon\]
    and 
    \[\HH(X_3\mid X_1,X_2)\geq \HH(X_2\mid X_1)+\log_2\frac{r-1}{r}-\varepsilon.\]

    To simplify the expressions, we define $a = 2^{\HH(\chi(X_1X_2)\mid X_1,X_3)}$ and $b=2^{\HH(\chi(X_1X_2),\chi(X_1X_3)\mid X_2,X_3)}$. Recall from the proof of \cref{claim:entropy-lower-bound} that
    \[\HH(Y_1\mid X_1,X_2,X_3) = \HH(X_3\mid X_1,X_2)-\log_2b,\]
    \[\HH(Y_2\mid X_1,X_2,X_3) = \HH(X_3\mid X_1,X_2)-\log_2a,\]
    and
    \[\HH(Y_3\mid X_1,X_2,X_3) = \HH(X_3\mid X_1,X_2).\]
    From this, we know that $p_1:p_2:p_3=b^{-1}:a^{-1}:1$.

    \paragraph{Proof of (a) and (b).}
    By the uniform bound, we know that $a\leq r-1$ and $b\leq (r-1)(r-2)$.
    Moreover, from the fact that $2^{\HH(Y\mid X_1,X_2,X_3)} = \sum_{i=1}^{3}2^{\HH(Y_i\mid X_1,X_2,X_3)}$, we see that
    \begin{align*}
        \HH(Y\mid X_1,X_2,X_3) =& \log_2\left(1+a^{-1}+b^{-1}\right)+\HH(X_3\mid X_1,X_2)\\
        \geq &\HH(X_1)+\log_2\left(\frac{r-2}{r-1}(1+a^{-1}+b^{-1})\right)-\varepsilon\\
        \geq & \HH(X_1)-\varepsilon.
    \end{align*}
    A similar argument applied to $\HH(Y'\mid X_1,X_2,X_3)$ gives $\HH(Y'\mid X_1,X_2,X_3)\geq \HH(X_2\mid X_1)-\varepsilon$. 
    This proves (a) and (b).

    \paragraph{Proof of (c).}
    To do this, we first bound $a$ and $b$.
    Indeed, since the marginal distribution of $Y$ is the same as $X_1$, we have
    \[\HH(X_1)+\log_2\left(\frac{r-2}{r-1}(1+a^{-1}+b^{-1})\right)-\varepsilon \leq \HH(Y\mid X_1,X_2,X_3)\leq \HH(Y)=\HH(X_1).\]
    Therefore,
    \[1+a^{-1}+b^{-1}\leq \frac{r-1}{r-2}+O(\varepsilon).\]
    It follows
    \[a^{-1}= \frac{1}{r-1}+O(\varepsilon),\,b^{-1}= \frac{1}{(r-1)(r-2)}+O(\varepsilon).\]
    Using $p_1:p_2:p_3=b^{-1}:a^{-1}:1$, we get $p_1=\frac{1}{(r-1)^2}+O(\veps)$, $p_2=\frac{r-2}{(r-1)^2}+O(\veps)$, $p_3=\frac{r-2}{r-1}+O(\veps)$.

    \paragraph{Proof of (d).} From the definition of $Y_2$ and the chain rule, we have
    \begin{align*}
        \HH(\chi(X_3Y_2)\mid X_1,X_2,X_3)=&\HH(\chi(X_2X_3)\mid X_1,X_3,\chi(X_1X_2))\\
        =&\HH(\chi(X_1X_2),\chi(X_2X_3)\mid X_1,X_3)-\HH(\chi(X_2X_3)\mid X_1,X_3)\\
        =&\log_2b-\log_2a\\
        \geq& \log_2(r-2)-O(r^2\veps).
    \end{align*}
    \paragraph{Proof of (e).} From the definition of $Y_3$, we have
    \begin{align*}
        \HH(\chi(X_1Y_3),\chi(X_2Y_3)\mid X_1,X_2,X_3)=&\HH(\chi(X_1X_3),\chi(X_2X_3)\mid X_1,X_2)\\
        =&\log b\geq \log_2((r-1)(r-2))-O(r^2\veps).
    \end{align*}

    \paragraph{Proof of (f).} Note that the tuple $(\chi(X_1X_2),\chi(X_1X_3),X_1)$ and the tuple $(\chi(X_1X_3),\chi(X_1X_3),X_1)$ have disjoint supports. Thus, by \cref{lem:mix} we have
    \[2^{\HH(\chi(X_1X_\mathbf{t}),\chi(X_1X_3),X_1)}\geq 2^{\HH(\chi(X_1X_2),\chi(X_1X_3),X_1)}+2^{\HH(\chi(X_1X_3),\chi(X_1X_3),X_1)}\]
    for some independent random index $\mathbf{t}\in \{2,3\}$. Dividing both sides by $2^{\HH(\chi(X_1X_3),X_1)}$, we get
    \[2^{\HH(\chi(X_1X_\mathbf{t})\mid\chi(X_1X_3),X_1)}\geq 2^{\HH(\chi(X_1X_2)\mid\chi(X_1X_3),X_1)}+1.\]
    Since $2^{\HH(\chi(X_1X_2)\mid\chi(X_1X_3), X_1)}\geq a$ and $\HH(\chi(X_1X_\mathbf{t}))\geq \HH(\chi(X_1X_\mathbf{t}\mid X_1)\mid\chi(X_1X_3),X_1)$, we can conclude that
    \[2^{\HH(\chi(X_1X_\mathbf{t})\mid X_1)}\geq a+1\geq r-1-O(r\veps)+1,\]
    and (f) follows as $(X_1,\chi(X_1X_2))$ and $(X_1,\chi(X_1X_3))$ are identically distributed.

    \paragraph{Proof of (g).} We have
    \begin{align*}
        \HH(Y_2\mid X_1,X_2,X_3)\geq& \HH(X_3\mid X_1,X_2)-\log_2(r-1)\\
        \geq& \HH(X_2\mid X_1)+\log_2\frac{r-1}{r}-\varepsilon-\log_2(r-1)\\
        =&\HH(X_2\mid X_1)+\log_2r-\varepsilon\\
        \geq& \HH(X_2\mid X_1)-\HH(\chi(X_1X_2)\mid X_1)-O(r\veps),
    \end{align*}
    where the first inequality follows from \cref{claim:entropy-lower-bound} and the last inequality follows from (f).
    By the fact that $\HH(X_2\mid X_1) = \HH(X_2,\chi(X_1X_2)\mid X_1)$ and the chain rule, this gives (g).
    \end{proof}
\end{lemma}

Now we derive \cref{lem:structure-easy-corollary} as an immediate corollary of \cref{lem:entropy_input}, which concludes the proof of \cref{thm:tight-construction}.
\begin{proof}[Proof of \cref{lem:structure-easy-corollary}]
    Since $\varepsilon=0$, we have $p_1 = \frac{1}{(r-1)^2}$, $p_2 = \frac{r-2}{(r-1)^2}$ and $p_3 = \frac{r-2}{r-1}$ by (c) in \cref{lem:entropy_input}.
    Also $\HH(Y'\mid X_1,X_2,X_3)\geq \HH(Y'\mid X_1)$ by \cref{lem:entropy_input} (b), showing that $Y'$ and $(X_2,X_3)$ are conditionally independent given $X_1$ by \cref{prop:drop}.
    Similarly, we know $\HH(Y\mid X_1,X_2,X_3)\geq \HH(Y)$ by \cref{lem:entropy_input} (a). By \cref{prop:drop} we know that the equality holds and $Y$ is independent of $(X_1,X_2,X_3)$.

    By \cref{lem:entropy_input} (d), we have $\HH(\chi(X_3Y_2)\mid X_1,X_2,X_3)\geq \log_2(r-2)$. On the other hand, we know that given any instance $(x_1,x_2,x_3)\in\supp(X_1,X_2,X_3)$, the support of $\chi(X_3Y_2)$ is contained in $[r]\backslash \{\chi(x_1x_2),\chi(x_1x_3)\}$, which has size $r-2$. Hence the reverse inequality also holds by \cref{prop:uniform} and the fact that $\HH(\chi(X_3Y_2)\mid X_1,X_2,X_3)$ is the expectation, as $x_1,x_2,x_3$ are sampled from the distribution of $(X_1,X_2,X_3)$, of $\HH(\chi(X_3Y_2)\mid X_1=x_1,X_2=x_2,X_3=x_3)$. This forces the equality to hold and the color $\chi(X_3Y_2)$ to distribute uniformly among $[r]\backslash \{\chi(X_1X_2),\chi(X_1X_3)\}$. Similarly, we have $\HH(\chi(X_1Y_3),\chi(X_2Y_3)\mid X_1,X_2)= \log_2((r-1)(r-2))$ by \cref{lem:entropy_input} (e) and the last part of the lemma follows.
\end{proof}

\subsection{Total variation and Pinsker's inequality}\label{subsec:total-variation}
For the proof of \cref{thm:stability_brief}, we can no longer deduce exact independence as we did in \cref{lem:structure-easy-corollary}.
Therefore we need the following tools that allow us to deal with approximate independence: we need tools that evaluate how close a joint distribution $P_{A,B}$ of random variables $A,B$ is to $P_{A,B'}$ where $B'$ is an independent resample of $B$.
From the point of view of entropy, this can be done by the Kullback--Leibler divergence, or the \emph{KL divergence}.

\begin{definition}[KL divergence]
    Let $A,A'$ be two random variables with $\supp(A)\subseteq \supp(A')$.
    Then define
    \[\kldiv{P_A}{P_{A'}} = \sum_{a\in \supp(A')}P_A(a)\log_2\left(\frac{P_{A}(a)}{P_{A'}(a)}\right).\]
\end{definition}

On the other hand, to produce a stability statement, it is more convenient for us to use the \emph{total variation} to evaluate how close two distributions are.

\begin{definition}[Total variation]
    Let $A,A'$ be two random variables.
    Define
    \[d_{\textup{TV}}(P_A,P_{A'}) =  \frac{1}{2}\sum_{a\in \supp(A)\cup \supp(A')}\abs{P_A(a)-P_{A'}(a)},\]
    or equivalently
    \[d_{\textup{TV}}(P_A,P_{A'}) = \max_{S\subseteq \supp(A)\cup \supp(A')}\abs{P_A(S)-P_A(S')}.\]
    If $d_{\textup{TV}}(P_A,P_{A'})\leq \varepsilon$, then we say $A$ is \emph{$\varepsilon$-close} to $A'$.
\end{definition}

To gain control of total variations from KL divergences, we use Pinsker's inequality.
Pinsker's inequality was originally proven in a weaker form by Pinsker \cite{Pin64}.
It was strengthened independently by Csisz\'ar \cite{Csi66}, Kullback \cite{Kul67} and Kemperman \cite{Kem69}, which is the version we state below.

\begin{proposition}[Pinsker's inequality]\label{prop:pinsker}
    Let $A,A'$ be two random variables with $\supp(A)\subseteq \supp(A').$
    Then
    \[d_{\textup{TV}}(P_A,P_{A'})\leq \sqrt{\frac{\log 2}{2}\kldiv{P_A}{P_{A'}}}.\]
\end{proposition}

We refer the readers to Lemma 2.5 in \cite[Section 2.4]{Tsy09} for a proof.

Before we show how we will use Pinsker's inequality in our setting, we need a simple proposition.

\begin{proposition}\label{prop:kl-chain-rule}
    Let $(A,B,B')$ be joint random variables with $\supp(A,B)\subseteq \supp(A,B')$.
    Then
    \[\kldiv{P_{A,B}}{P_{A,B'}} = \sum_{a\in \supp(A)}P_A(a)\kldiv{P_{B\mid A=a}}{P_{B'\mid A=a}}.\]
\end{proposition}
\begin{proof}
    The right hand side is simply
    \[\sum_{a\in \supp(A)}\sum_{b\in \supp(B'\mid A=a)}P_A(a)P_{B\mid A=a}(b)\log_2\left(\frac{P_{B\mid A=a}(b)}{P_{B'\mid A=a}(b)}\right),\]
    which by the chain rule for conditional probability is simply
    \[\sum_{(a,b)\in \supp(A,B')}P_{A,B}(a,b)\log_2\left(\frac{P_{A,B}(a,b)}{P_{A,B'}(a,b)}\right).\]
    The identity now follows from the definition of $\kldiv{P_{A,B}}{P_{A,B'}}$.
\end{proof}

Using this, we first show that if the entropy of $B,C$ is almost as large as the entropy of $B',C$ where $B'$ is an independent resample, this means that $B,C$ is close to $B',C$ in total variation.
We will need a conditional statement stated as follows.

\begin{corollary}\label{cor:mutual-info-small}
    Let $A,B,C$ be joint random variables, $B'$ be a conditionally independent resample of $B$ given $A$, and $\varepsilon>0$ be some real number.
    If $\HH(A,B,C)\geq \HH(A,B',C)-\varepsilon$, then $(A,B,C)$ is $O(\veps^{1/2})$-close to $(A,B',C)$.
\end{corollary}
\begin{proof}
    We first note that for any two random variables $X,Y$, if $X'$ is an independent resample of $X$, then
    \begin{align*}\kldiv{P_{X,Y}}{P_{X',Y}} = &\sum_{x,y}P_{X,Y}(x,y)\log_2\frac{P_{X,Y}(x,y)}{P_X(x)P_Y(y)} \\ = &\sum_x-P_X(x)\log_2P_X(x)+\sum_y-P_Y(y)\log_2P_Y(x)\\&-\sum_{x,y}-P_{X,Y}(x,y)\log_2P_{X,Y}(x,y) \\=& \HH(X',Y)-\HH(X,Y).
    \end{align*}
    Now by \cref{prop:kl-chain-rule}, it is clear that
    \begin{align*}\kldiv{P_{A,B,C}}{P_{A,B',C}} =& \sum_aP_A(a)\kldiv{P_{B,C\mid A=a}}{P_{B',C\mid A=a}} \\=& \sum_aP_A(a)\left(\HH(B',C\mid A=a)-\HH(B,C\mid A=a)\right)\\=&\HH(A,B',C)-\HH(A,B,C)\leq \varepsilon.
    \end{align*}
    Therefore we can apply Pinsker's inequality (\cref{prop:pinsker}) to conclude.
\end{proof}

Similarly, if $B$ has entropy almost as large as $\log_2\abs{\supp(B)}$, then $B$ must be close to uniform.
We will once again need a conditional version of this.

\begin{corollary}\label{cor:almost-uniform}
    Let $A,B$ be a random variables, $n\in \NN$ be some positive integer and $\varepsilon$ be some real number.
    Suppose that for all $a\in \supp(A)$ we have a set $S_a$ of size $n$ so that $\supp(B\mid A=a)\subseteq S_a$.
    Given $A=a$, let $B'$ be sampled uniformly from $S_a$.
    If $\HH(B\mid A)\geq \log_2n-\varepsilon$, then $(A,B)$ is $O(\veps^{1/2})$-close to $(A,B')$.
\end{corollary}
\begin{proof}
    We first note that for any random variable $X$ and any superset $S$ of $\supp(X)$, if $X'$ is uniform on $S$, then
    \[\kldiv{P_X}{P_{X'}} = \sum_xP_X(x)\left(\log_2P_X(x)+\log_2\abs{S}\right) = \log_2\abs{S}-\HH(X).\]
    By \cref{prop:kl-chain-rule}, we know that
    \begin{align*}\kldiv{P_{A,B}}{P_{A,B'}} =& \sum_aP_A(a)\kldiv{P_{B\mid A=a}}{P_{B'\mid A=a}}\\ =& \sum_{a}P_A(a)\left(\log_2\abs{S_a}-\HH(B\mid A=a)\right) \\=& \log_2 n - \HH(B\mid A).
    \end{align*}
    Therefore we can conclude using Pinsker's inequality (\cref{prop:pinsker}).
\end{proof}

For a major portion of the proof, we will use total variations to measure closeness of random variables.
In the remaining of this subsection, we introduce several tools that allow us to manipulate total variations.
We start with the following standard fact.

\begin{proposition}[Maximal coupling]
    Suppose that $A,A'$ are random variables.
    Then there exists a coupling $(A,A')$ of the two random variables so that
    \[\PP(A\neq A') = d_{\textup{TV}}(P_A,P_{A'}).\]
\end{proposition}
\begin{proof}
    Consider any coupling $(A,A')$ so that $P_{A,A'}(a,a) = \min\{P_A(a),P_{A'}(a)\}$ for any $a$ lying in $\supp(A)\cap \supp(A')$.
    Note that such coupling is possible as $P_{A,A'}(a,a)$ is bounded above by both $P_A(a)$ and $P_{A'}(a)$.
    Now with this coupling, 
    \begin{align*}
        \PP(A\neq A')=&1-\sum_{a\in \supp(A)\cup \supp(A')}\min\{P_A(a),P_{A'}(a)\}\\
        =& \sum_{a\in \supp(A)\cup \supp(A')}\left(\frac{P_A(a)+P_{A'}(a)}{2}-\min\{P_A(a),P_{A'}(a)\}\right)\\
        =& \frac{1}{2}\sum_{a\in \supp(A)\cup \supp(A')}\abs{P_A(a)-P_{A'}(a)}\\
        =&d_{\textup{TV}}(P_A, P_{A'}),
    \end{align*}
    as desired.
\end{proof}

The following two propositions deal with how total variations behave with mixtures.

\begin{proposition}\label{prop:mixing-variable}
Let $A_1,\ldots, A_k$ and $B_1,\ldots, B_k$ be random variables, and let $\bfi,\bfj\in [k]$ be independent random indices.
Assume that $A_i$ is $\varepsilon$-close to $B_i$ for all $i\in [k]$ and $\bfi$ is $\varepsilon$-close to $\bfj$.
Then $A_{\bfi}$ is $2\varepsilon$-close to $B_{\bfj}$.
\end{proposition}
\begin{proof}
    Consider any coupling $(A_1,\ldots, A_k, B_1,\ldots, B_k)$ so that $\PP(A_i\neq B_i)\leq \varepsilon$ for any $i\in [k]$.
    Also consider any coupling $(\bfi, \bfj)$ so that $\PP(\bfi\neq \bfj)\leq \varepsilon$.
    Then 
    \[\PP(A_{\bfi}\neq B_{\bfj})\leq \PP(\bfi\neq \bfj) + \PP(A_{\bfi}\neq B_{\bfi})\leq \varepsilon+\sum_{i=1}^{k}P_{\bfi}(i)\PP(A_i\neq B_i)\leq 2\varepsilon,\]
    as desired.
\end{proof}

\begin{proposition}\label{prop:unmixing-variable}
    Let $A_1,A_2,B_1,B_2$ be joint random variables and $\bfi\in [2]$ be an independent random index.
    Let $p_1 = \PP(\bfi = 1)$ and let $A=A_{\bfi}, B = B_{\bfi}$. 
    Assume that $A_2$ is $\varepsilon$-close to $B_2$ and $A$ is $\varepsilon$-close to $B$.
    Then $A_1$ is $2p_1^{-1}\varepsilon$-close to $B_1$.
\end{proposition}
\begin{proof}
    We know that for any $a\in \supp(A)$, we have
    \[P_{A_1}(a) = \frac{P_A(a)-(1-p_1)P_{A_2}(a)}{p_1}\]
    and similarly for $P_{B_1}$.
    Therefore
    \begin{align*}&\frac{1}{2}\sum_{a\in \supp(A_1)\cup \supp(B_1)}\abs{P_{A_1}(a)-P_{B_1}(a)}\\ \leq & \frac{1}{2p_1}\sum_{a\in \supp(A_1)\cup \supp(B_1)}\left(\abs{P_{A}(a)-P_{B}(a)}+\abs{P_{A_2}(a)-P_{B_2}(a)}\right)\\\leq &2p_1^{-1}\varepsilon,
    \end{align*}
    as desired.
\end{proof}

For our applications, it is sometimes more convenient to work with two different sequence of probabilities for mixtures.

\begin{corollary}\label{cor:advanced-unmixing-variable}
    Let $A_1,A_2,B_1,B_2$ be joint random variables and $\bfi,\bfj\in [2]$ be an independent random index.
    Let $p_1 = \PP(\bfi = 1)$, $q_1 = \PP(\bfj=1)$ and let $A=A_{\bfi}, B = B_{\bfj}$. 
    Assume that $A_2$ is $\varepsilon$-close to $B_2$ and $A$ is $\varepsilon$-close to $A'$.
    Then $A_1$ is $\varepsilon'$-close to $B_1$ with $\varepsilon'=2p_1^{-1}(\varepsilon+\abs{p_1-q_1})$.
\end{corollary}
\begin{proof}
    We know that $\bfi$ and $\bfj$ are $\delta\eqdef \abs{p_1-q_1}$-close to each other, showing that $B_{\bfi}$ is $\delta$-close to $B$.
    Therefore $A$ and $B_{\bfi}$ are $(\varepsilon+\delta)$-close to each other.
    Now we may apply \cref{prop:unmixing-variable} to get that $A_1$ and $B_1$ are $\varepsilon'$-close to each other where
    \[\varepsilon' = 2p_1^{-1}(\varepsilon+\delta) = 2p_1^{-1}(\varepsilon+\abs{p_1-q_1}).\qedhere\]
\end{proof}

We will also frequently resample a random variable conditionally independently, so it is important how total variation behaves for conditionally independent random variables.

\begin{proposition}\label{prop:total-variation-gluing}
    Let $A,B,C$ be joint random variables so that $A,C$ are conditionally independent given $B$.
    Let $A',B',C'$ be analogously defined.
    Finally, let $\varepsilon_1,\varepsilon_2>0$ be any real numbers.
    If $(A,B)$ is $\varepsilon_1$-close to $(A',B')$ and $(B,C)$ is $\varepsilon_2$-close to $(B',C')$, then $(A,B,C)$ is $(\varepsilon_1+\varepsilon_2)$-close to $(A',B',C')$.
\end{proposition}

To prove this, we need the following lemma.

\begin{lemma}\label{lem:extend-coupling}
    Let $A,B$ be joint random variables, and similarly for $A',B'$.
    Then any coupling of $B,B'$ with $\PP(B\neq B') = d_{\textup{TV}}(P_B, P_{B'})$ can be extended to a coupling of $(A,B),(A',B')$ with $\PP((A,B)\neq (A',B')) = d_{\textup{TV}}(P_{A,B},P_{A',B'})$.
\end{lemma}
\begin{proof}
    If $(A,B)$ and $(A',B')$ have disjoint supports, then there are nothing to prove.
    Otherwise, we know that $d_{\textup{TV}}(P_{A,B}, P_{A',B'})<1$.
    For simplicity, let $s = 1-d_{\textup{TV}}(P_{A,B}, P_{A',B'})>0$.
    Let $(A^*,B^*)$ be such that 
    \[P_{A^*,B^*}(a,b) = s^{-1}\min\{P_{A,B}(a,b),P_{A',B'}(a,b)\}.\]
    Then it is clear that $(A,B)$ is a mixture of $(A^*,B^*)$ and some random pair $(C,D)$ with probabilities $(s,1-s)$, and $(A',B')$ is a mixture of $(A^*,B^*)$ and some random pair $(C',D')$ with probabilities $(s,1-s)$.

    Now for any $b\in \supp(B)\cap \supp(B')$, we see that 
    \[sP_{B^*}(b) = \sum_{a}\min\{P_{A,B}(a,b),P_{A',B'}(a,b)\}\leq \min\{P_B(b), P_{B'}(b)\}=P_{B,B'}(b,b).\]
    This shows that $(B,B')$ is a mixture of $(B^*,B^*)$ and some other $(E,E')$ with probabilities $(s,1-s)$.
    It is clear that the marginal distribution of $E$ is the same as $D$, and the marginal distribution of $E'$ is the same as $D'$.
    Therefore we may sample $(C,D,C',D')$ jointly so that the marginal distribution of $(D,D')$ is the same as $(E,E')$.
    Now take $(A,B,A',B')$ to be the mixture of $(A_*,B_*,A_*,B_*)$ and $(C,D,C',D')$ with probabilities $(s,1-s)$, and this is the desired coupling.
\end{proof}

\begin{proof}[Proof of \cref{prop:total-variation-gluing}]
    Let $(B,B')$ be any coupling so that $\PP(B\neq B') = d_{\textup{TV}}(P_B, P_{B'})$.
    Apply \cref{lem:extend-coupling} to extend to a coupling $(A,B,A',B')$ and a coupling $(B,C,B',C')$.
    Therefore we get a coupling $(A,B,C_*,A',B',C_*')$ where $(A,A')$ and $(C_*,C_*')$ are conditionally independent given $(B,B')$.
    Therefore the marginal distribution of $(A,B,C_*)$ is the same as $(A,B,C)$, and the marginal distribution of $(A',B',C_*')$ is the same as $(A',B',C')$, showing that we indeed get a coupling $(A,B,C,A',B',C')$.
    This coupling satisfies that $\PP((A,B)\neq (A',B'))\leq \varepsilon_1$ and also $\PP((B,C)\neq (B',C'))\leq \varepsilon_2$.
    By the union bound, this means that $\PP((A,B,C)\neq (A',B',C'))\leq \varepsilon_1+\varepsilon_2$, as desired.
\end{proof}

Finally, at some point of the argument, we would like to fix some variables while making sure that some other random variables are still close to each other.
To make sure that it is possible to do so, we need the following simple corollary of the Markov's inequality.

\begin{proposition}\label{prop:total-variation-markov}
    Let $A,B,B'$ be joint random variables, and let $p\in (0,1)$.
    Let $S_p\subseteq \supp(A)$ be the set of $a\in \supp(A)$ with
    \[d_{\textup{TV}}(P_{B\mid A=a},P_{B'\mid A=a})< p^{-1}d_{\textup{TV}}(P_{A,B}, P_{A,B'}).\]
    Then $P_A(S_p)\geq 1-p$.
\end{proposition}
\begin{proof}
    Let $f(a) = d_{\textup{TV}}(P_{B\mid A=a},P_{B'\mid A=a})$ for any $a\in \supp(A)$.
    Then it is clear by definition that 
    \[d_{\textup{TV}}(P_{A,B},P_{A,B'}) = \EE_A f(A).\]
    Moreover, it is clear that $S_p$ is the set of elements $a$ with $f(a)<p^{-1}\EE_Af(A)$.
    By Markov's inequality, we know that $P_A(\supp(A)\backslash S_p)\leq p$, and so $P_A(S_p)\geq 1-p$.
\end{proof}

\subsection{Stability}\label{subsec:stability}
We will follow a similar strategy as in the proof of \cref{thm:tight-construction}, although the structure we may extract is going to be weaker. 
To be specific, we will show the following.
\begin{theorem}\label{thm:strong-stability}
    Let $G$ be a simple $r$-edge-colored graph, and let $(X_1,X_2,X_3)$ be a random rainbow triangle in $G$ following a symmetric distribution.
    Suppose that
    \[2\HH(X_1,X_2,X_3)= 3\HH(X_1,X_2)+\log_2\frac{r-2}{r}-\veps\]
    with $\veps<\alpha r^{-16}$ for some absolute constant $\alpha>0$ small enough.
    Then there exist random variables $A_1,\dots,A_{(r-1)^2}\in V(G)$ such that 
\begin{enumerate}
    \item their supports $V_i\eqdef \supp(A_i)$ are mutually disjoint,
    \item for each $1\leq i\neq j\leq (r-1)^2$, there is a color $c_{ij}$ such that $A_{i}A_{j}$ is an edge with color $c_{ij}$ with probability at least $1-O(r^6\veps^{1/2})$,
    \item the complete graph on $[(r-1)^2]$ with color $c_{ij}$, which we call $H$, is a graph induced by an affine plane of order $(r-1)$, and
    \item $(X_1,X_2,X_3)$ is $O(r^2\veps^{1/2})$-close to $(A_{\bfi},A_{\bfj},A_{\bfk})$, where $(\bfi,\bfj,\bfk)$ is chosen from $[(r-1)^2]^3$ uniformly at random such that $c_{\bfi\bfj},c_{\bfi\bfk},c_{\bfj\bfk}$ are all distinct.
\end{enumerate}

\end{theorem}
It is clear that if we take $X_1,X_2,X_3$ to be a uniform rainbow triangle, then this theorem implies  (the contrapositive of) \cref{cor:weak-stability}.
In the remaining of the subsection, we give a proof of \cref{thm:strong-stability}.

We still begin by applying \cref{lem:entropy_input}, which gives us parameters $p_1 = \frac{1}{(r-1)^2}+O(\veps)$, $p_2 = \frac{r-2}{(r-1)^2}+O(\veps)$, $p_3 = \frac{r-2}{r-1}+O(\veps)$, and random variables $Y_1,Y_2,Y_3,Y,Y'$ that we will use throughout the subsection.

We would like to still show a version of \cref{lem:structure-easy-corollary}, but now we cannot expect the full independence that held in the previous setting.
Instead, we will apply the propositions proven in the previous subsection to prove approximate independence.
To this end, we do the following sampling.
Let $(Z_1,Z,Z')$ be a conditionally independent resampling of $(Y_1,Y,Y')$ given just $X_1$.
Also let $W$ be an independent resampling of $Y$.

\begin{lemma}\label{lem:Y-approx-independence}
    The tuples $(X_1,X_2,X_3,Y)$, $(X_1,X_2,X_3,Z)$ and $(X_1,X_2,X_3,W)$ are $O(\veps^{1/2})$-close to each other.
\end{lemma}
\begin{proof}
    By \cref{lem:entropy_input} and the fact that $X_1$ and $Y$ are identically distributed, we know that
    \[\HH(X_1,X_2,X_3,Y)\geq \HH(Y)+\HH(X_1,X_2,X_3) -\veps = \HH(X_1,X_2,X_3,W)-\veps,\]
    and so we know that $(X_1,X_2,X_3,Y)$ is $O(\veps^{1/2})$-close to $(X_1,X_2,X_3,W)$ by \cref{cor:mutual-info-small} applied with $A=\emptyset$, $B = Y$ and $C=(X_1,X_2,X_3)$.
    Similarly, we know that
    \[\HH(X_1,X_2,X_3,Y)\geq \HH(Y\mid X_1)+\HH(X_1,X_2,X_3) -\veps=\HH(X_1,X_2,X_3,Z)-\veps,\]
    and so \cref{cor:mutual-info-small} applies with $A=X_1$, $B = Y$ and $C= (X_2,X_3)$.
\end{proof}

\begin{lemma}\label{lem:Yprime-approx-independence}
    The tuples $(X_1,X_2,X_3,Y')$ and $(X_1,X_2,X_3,Z')$ are $O(\veps^{1/2})$-close to each other.
\end{lemma}
\begin{proof}
    By \cref{lem:entropy_input} and the fact that $(X_1,Y')$ is identically distributed as $(X_1,X_2)$, we know that
    \[\HH(X_1,X_2,X_3,Y') \geq \HH(Y'\mid X_1)+\HH(X_1,X_2,X_3)-\veps = \HH(X_1,X_2,X_3,Z'),\]
    and so we may apply \cref{cor:mutual-info-small} with $A=X_1$, $B=Y'$ and $C=(X_2,X_3)$.
\end{proof}

Next, sample $(\bfc_1,\bfc_2)$ conditionally independently and uniformly from 
\[\{(c_1,c_2)\in [r]^2\mid c_1,c_2,\chi(X_1X_2)\textup{ all distinct}\}\]
given $\chi(X_1X_2)$. 
Also sample $\bfc_0$ from $[r]\backslash \{\chi(X_1X_2), \chi(X_1X_3)\}$ conditionally independently and uniformly given $\chi(X_1X_2), \chi(X_1X_3)$.

\begin{lemma}\label{lem:color-approx-independence}
    The tuples $(X_1,X_2,X_3,\chi(X_3Y_2))$ and $(X_1,X_2,X_3,\bfc_0)$ are $O(r\veps^{1/2})$-close to each other.
    The tuples $(X_1,X_2,X_3,\chi(X_1Y_3),\chi(X_2Y_3))$ and $(X_1,X_2,X_3,\bfc_1,\bfc_2)$ are $O(r\veps^{1/2})$-close to each other as well.
\end{lemma}
\begin{proof}
    Recall from \cref{lem:entropy_input} that
    \begin{align*}\HH(\chi(X_1Y_3),\chi(X_2Y_3)\mid X_1,X_2,X_3) \geq \log_2((r-1)(r-2))-O(r^2\veps).
    \end{align*}
    Also note that for any $(x_1,x_2,x_3)\in \supp(X_1,X_2,X_3)$, we have that $(\chi(X_1Y_3),\chi(X_1Y_3))$ is in the set
    \[\{c_1,c_2\in [r]\backslash \{\chi(x_1x_2)\}\mid c_1\neq c_2\},\]
    which has size $(r-1)(r-2)$.
    By \cref{cor:almost-uniform}, we know that $(X_1,X_2,X_3,\chi(X_2Y_3),\chi(X_1Y_3))$ is $O(r\veps^{1/2})$-close to $(X_1,X_2,X_3,\bfc_1,\bfc_2)$, as desired.
    To show that the tuples $(X_1,X_2,X_3,\chi(X_3Y_2))$ and $(X_1,X_2,X_3,\bfc_0)$, the same argument holds except that now one uses
    \[\HH(\chi(X_3Y_2)\mid X_1,X_2,X_3)\geq \log_2(r-2)-O(r^2\veps).\qedhere\]
\end{proof}

Now, we would again like to show that $Y_1$ is approximately conditionally independent of $X_2,X_3$ given $X_1$.
This will be a simple corollary of \cref{prop:unmixing-variable}.

\begin{lemma}\label{lem:Y1-approx-independence}
    The tuples $(X_1,X_2,X_3,Y_1)$ and $(X_1,X_2,X_3,Z_1)$ are $O(r^2\veps^{1/2})$-close to each other.
\end{lemma}
\begin{proof}
    Note that $Y$ is a mixture of $Y_1$ and $Y'$ with $\PP(Y=Y_1) = p_1 = \Theta(r^{-2})$ (where we use $\veps = o(r^{-2})$).
    Also $Z$ is a mixture of $Z_1$ and $Z'$ with $\PP(Z=Z_1)=p_1$ as well.
    Therefore by \cref{lem:Y-approx-independence,lem:Yprime-approx-independence} and \cref{prop:unmixing-variable} we see that $(X_1,X_2,X_3,Y_1)$ is $O(r^2\veps^{1/2})$-close to $(X_1,X_2,X_3,Z_1)$.
\end{proof}

To apply these approximate independences is nonetheless trickier than before.
The main difficulty is that the partition generated in the argument of \cref{lem:structure-X1-Y1} did not depend on the rainbow triangle we started with, but now it can vary a little bit even for most of the rainbow triangles.
To circumvent it, we would need to first extract the key step from the argument and prove an approximate version of it.
To this end, for each $i\in[3]$, take $Y_i^*$ to be a conditionally independent resample of $Y_i$ given $X_1,X_2,X_3,\chi(Y_iX_{i+1}), \chi(Y_iX_{i+2})$.
Then let $Y^*$ be such that $(X_1,X_2,X_3,Y,Y^*)$ is the mixture of $(X_1,X_2,X_3,Y_i,Y_i^*)_{i=1}^{3}$ with probabilities $(p_1,p_2,p_3)$.
(We remark that this is precisely the $Y^*$ we constructed in the proof of \cref{lem:structure-X1-Y1} when $X_1X_2X_3$ is a uniform rainbow triangle.)
It is easy to see that $(Y,Y^*)$ is identically distributed as $(X_1,Y_1)$, and the key claim is that it is approximately independent of $(X_1,X_2,X_3)$.
Recall that $W$ is an independent resample of $Y$, which is the same as an independent resample of $X_1$.
Let $W^*$ be sampled conditionally independently given $W$ so that $(W,W^*)$ is identically distributed as $(X_1,Y_1)$.

\begin{lemma}\label{lem:Ystar-approx-independence}
    The tuples $(X_1,X_2,X_3,Y,Y^*)$ and $(X_1,X_2,X_3,W,W^*)$ are $O(r^2\veps^{1/2})$-close to each other.
\end{lemma}
\begin{proof}
    For each $i\in[3]$, let $W_i$ be sampled conditionally independently given $Y_i$ so that the distribution of $(Y_i,W_i)$ is the same as $(X_1,Y_1)$.
    Then by \cref{lem:Y1-approx-independence}, we know that $(Y_i,X_{i+1},X_{i+2},Y^*_i)$ is $O(r^2\veps^{1/2})$-close to $(Y_i,X_{i+1},X_{i+2},W_i)$.
    Since $X_i$ and $(Y_i, Y^*_i)$ are conditionally independent given $X_{i+1}, X_{i+2}$, and $X_i$ and $(Y_i, W_i)$ are conditionally independent given $X_{i+1},X_{i+2}$, by \cref{prop:total-variation-gluing} we know that $(X_1,X_2,X_3,Y_i,Y_i^*)$ and $(X_1,X_2,X_3,Y_i, W_i)$ are $O(r^2\veps^{1/2})$-close to each other.

    Now let $(X_1,X_2,X_3,Y,\widetilde{W})$ be the mixture of $(X_1,X_2,X_3,Y_i,W_i)_{i=1}^3$ with probabilities $(p_1,p_2,p_3)$.
    By \cref{prop:mixing-variable}, we know that $(X_1,X_2,X_3,Y,Y^*)$ is $O(r^2\veps^{1/2})$-close to $(X_1,X_2,X_3,Y,\widetilde{W})$.
    It remains to show that $(X_1,X_2,X_3,Y,\widetilde{W})$ is close to $(X_1,X_2,X_3,W,W^*)$.

    Note that $(X_1,X_2,X_3)$ and $\Tilde{W}$ are conditionally independent given $Y$, and $(X_1,X_2,X_3)$ and $W^*$ are conditionally independent given $W$.
    Also note that $(Y,\Tilde{W})$ and $(W,W^*)$ are both identically distributed as $(X_1,Y_1)$.
    Therefore by \cref{lem:Y-approx-independence} and \cref{prop:total-variation-gluing} applied with $A=A'=(X_1,X_2,X_3)$, $B=Y$, $B'=W$, $C=\widetilde{W}$ and $C' = W^*$, we know that $(X_1,X_2,X_3,Y,\widetilde{W})$ is $O(\veps^{1/2})$-close to $(X_1,X_2,X_3,W,W^*)$.
    We are thus done by the triangle inequality.
\end{proof}

We will now fix a good rainbow triangle so that intuitively we can fix the sets $V_1,\ldots, V_{(r-1)^2}$ with similar effects as in \cref{lem:structure-X1-Y1}.
In particular, by \cref{prop:total-variation-markov} and \cref{lem:color-approx-independence,lem:Ystar-approx-independence}, there exists a rainbow triangle $x=(x_1,x_2,x_3)$ so that the following properties hold.
\begin{enumerate}
    \item $(Y_x, Y_x^*)\eqdef (Y,Y^*\mid (X_1,X_2,X_3)=(x_1,x_2,x_3))$ is $O(r^2\veps^{1/2})$-close to $(W,W^*)$.
    \item $(\chi(X_3Y_2)\mid (X_1,X_2,X_3) = (x_{1},x_{2},x_{3}))$ is $O(r\veps^{1/2})$-close to $(\bfc_0\mid (X_1,X_2,X_3)=(x_{1},x_{2},x_{3}))$.
    \item $(\chi(X_1Y_3), \chi(X_2Y_3)\mid (X_1,X_2,X_3) = (x_{1},x_{2},x_{3}))$ is $O(r\veps^{1/2})$-close to $(\bfc_1,\bfc_2\mid (X_1,X_2,X_3)=(x_{1},x_{2},x_{3}))$.
\end{enumerate}
With this choice of rainbow triangle, we may now define the analog of the sets $V_1,\ldots, V_{(r-1)^2}$.
Instead of sets, we will define random variables $A_1,\ldots, A_{(r-1)^2}$ with disjoint supports.
These will be defined as follows.
\begin{enumerate}
    \item The random variable $A_1$ is sampled from the distribution $(Y_1\mid (X_1,X_2,X_3)=(x_1,x_2,x_3))$ independently.
    \item The random variables $A_2,\ldots, A_{r-1}$ are sampled from the distributions $(Y_2\mid (X_1,X_2,X_3)=(x_1,x_2,x_3), \chi(X_3Y_2)=c)$ independently when $c$ ranges through $[r]\backslash \{\chi(x_1x_2),\chi(x_1x_3)\}$.
    \item The random variables $A_r,\ldots, A_{(r-1)^2}$ are sampled from the distributions $(Y_3\mid (X_1,X_2,X_3)=(x_1,x_2,x_3), (\chi(X_1Y_3),\chi(X_2Y_3)) = (c_1,c_2))$ independently when $(c_1,c_2)$ ranges through pairs in $\left([r]\backslash \{\chi(x_1x_2)\}\right)^2$ with $c_1\neq c_2$.
\end{enumerate}
We will also set $V_i = \supp(A_i)$ for each $i\in[(r-1)^2]$.
We will show that this choice of random variables satisfies all four requirements in the statement of \cref{thm:strong-stability}.
The first item is clear, and it remains to prove the other three.

Having set up the random variables $A_1,\ldots, A_{(r-1)^2}$, we may now prove the analog of \cref{lem:structure-X1-Y1}.
Let $\bfi\in [(r-1)^2]$ be an independent uniform random variable on $[(r-1)^2]$.
Also for each $i\in[(r-1)^2]$, let $A_i^*$ be an independent resampling of $A_i$.

\begin{lemma}\label{lem:X1-Y1-approx-distribution}
    We have that $(X_1,Y_1)$ is $O(r^2\veps^{1/2})$-close to $(A_{\bfi}, A_{\bfi}^*)$.
\end{lemma}
\begin{proof}
    Recall that $(X_1,Y_1)$ is identically distributed as $(W,W^*)$, which is $O(r^2\veps^{1/2})$-close to $(Y_x, Y_x^*)$.
    Therefore it suffices to show that $(Y_x, Y_x^*)$ is $O(r^2\veps^{1/2})$-close to $(A_{\bfi}, A_{\bfi}^*)$.

    Note that $(Y_x, Y_x^*)$ and $(A_{\bfi},A_{\bfi}^*)$ are both mixtures of $(A_i,A_i^*)_{i=1}^{(r-1)^2}$, though with different probabilities.
    We know that $(A_{\bfi},A_{\bfi}^*)$ is a mixture of $(A_i, A_i^*)_{i=1}^{(r-1)^2}$ with uniform probabilities $((r-1)^{-2},\ldots, (r-1)^{-2})$.
    Suppose that $(Y_x, Y_x^*)$ is a mixture of $(A_i, A_i^*)_{i=1}^{(r-1)^2}$ with probabilities $(q_i)_{i=1}^{(r-1)^2}$.
    By definition, we know that $p_1=q_1$, $p_2 = q_2+\cdots+q_{r-1}$ and $p_3 = q_r+\cdots+q_{(r-1)^2}$.
    We will show that $(q_1,\ldots, q_{(r-1)^2})$ is close to uniform.

    We begin by showing that $q_r,\ldots, q_{(r-1)^2}$ are almost the same.
    In fact, the distribution with masses $p_3^{-1}q_r$, $\ldots$, $p_3^{-1}q_{(r-1)^2}$ is precisely the distribution of $(\chi(X_1Y_3),\chi(X_2Y_3)\mid (X_1,X_2,X_3)=(x_1,x_2,x_3))$, which is $O(r\veps^{1/2})$-close to uniform.
    This shows that $\bfi_3'$ is $O(r\veps^{1/2})$-close to $\bfi_3$, where $\PP(\bfi_3'=i)=p_3^{-1}q_i$ for $i=r,\ldots, (r-1)^2$ and $\bfi_3$ is uniform on $\{r,\ldots, (r-1)^2\}$.

    Similarly, the distribution with masses $p_2^{-1}q_2,\ldots, p_2^{-1}q_{r-1}$ is precisely the distribution of the color $(\chi(X_3Y_2)\mid (X_1,X_2,X_3)=(x_1,x_2,x_3))$, which is $O(r\veps^{1/2})$-close to $(\bfc_0\mid (X_1,X_2,X_3)=(x_1,x_2,x_3))$, which is uniform on $[r]\backslash \{\chi(x_1x_2),\chi(x_1x_3)\}$.
    This shows that $\bfi_2'$ and $\bfi_2$ are $O(r\veps^{1/2})$-close to each other, where $\PP(\bfi_2'=i)=p_2^{-1}q_i$ for $i=2,\ldots, r-1$ and $\bfi_2$ is uniform on $\{2,\ldots, r-1\}$.

    Finally, let $\bfi_1=\bfi_1'=1$ deterministically.
    Then $\bfi$ is the mixture of $\bfi_1,\bfi_2,\bfi_3$ with probabilities $(\frac{1}{(r-1)^2}, \frac{r-2}{(r-1)^2}, \frac{r-2}{r-1})$.
    Set $\bfi'$ to be the mixture of $\bfi_1',\bfi_2',\bfi_3'$ with probabilities $(p_1,p_2,p_3)$.
    Recall that the distribution with masses $\frac{1}{(r-1)^2}, \frac{r-2}{(r-1)^2}, \frac{r-2}{r-1}$ is $O(\veps)$-close to the one with masses $p_1,p_2,p_3$.
    Therefore by \cref{prop:mixing-variable}, we know that $\bfi$ is $O(r\veps^{1/2})$-close to $\bfi'$.
    Now as $(Y_x,Y_x^*)=(A_{\bfi'}, A_{\bfi'}^*)$, we can deduce that $(A_{\bfi}, A_{\bfi}^*)$ is $O(r\veps^{1/2})$-close to $(Y_x,Y_x^*)$ using \cref{prop:mixing-variable} again.
\end{proof}

Similar to \cref{lem:structure-X1-Yprime}, we may now study the distribution of $(X_1,X_2)$ and use it to conclude that $G$ is mostly monochromatic and complete between $V_i$ and $V_j$.
As the argument is more involved in this setting, we split this step into two lemmas as follows.
Below, let $\bfj$ be sampled conditionally independently given $\bfi$ and uniformly from $[(r-1)^2]\backslash \bfi$.

\begin{lemma}\label{lem:X1-X2-approx-distribution}
    The pair $(X_1,X_2)$ is $O(r^2\veps^{1/2})$-close to $(A_{\bfi}, A_{\bfj})$.
\end{lemma}
\begin{proof}
    As in \cref{lem:structure-X1-Yprime}, we will use the distributions of $(X_1,Y)$ and $(X_1,Y_1)$ to study the distribution of $(X_1,Y')$.
    Recall from \cref{lem:Y-approx-independence} that $(X_1,Y)$ is $O(\veps^{1/2})$-close to $(X_1,W)$, and \cref{lem:X1-Y1-approx-distribution} shows that $X_1$ is $O(r^2\veps^{1/2})$-close to $A_{\bfi}$.
    Let $\bfi^*$ be an independent resampling of $\bfi$.
    As $X_1,W$ are i.i.d. and $A_{\bfi}, A_{\bfi^*}^*$ are i.i.d. as well, by \cref{prop:total-variation-gluing} we know that $(X_1,W)$ is $O(r^2\veps^{1/2})$-close to $(A_{\bfi}, A_{\bfi^*}^*)$.
    Therefore $(X_1,Y)$ is $O(r^2\veps^{1/2})$-close to $(A_{\bfi}, A_{\bfi^*}^*)$.

    Now notice that $(X_1,Y)$ is a mixture of $(X_1,Y_1)$ and $(X_1,Y')$ with probabilities $(p_1,p_2+p_3)$.
    Also notice that $(A_{\bfi}, A_{\bfi^*}^*)$ is a mixture of $(A_{\bfi}, A_{\bfi}^*)$ and $(A_{\bfi}, A_{\bfj}^*)$ with probabilities $(\frac{1}{(r-1)^2}, \frac{r(r-2)}{(r-1)^2})$.
    Since $p_2+p_3 = \frac{r(r-2)}{(r-1)^2}+O(\veps)$, by \cref{cor:advanced-unmixing-variable}, we know that $(X_1,Y')$ is $O(r^2\veps^{1/2})$-close to $(A_{\bfi}, A_{\bfj}^*)$.
    Since $A_{\bfi}$ and $A_{\bfj}$ are independent, we see that $(A_{\bfi},A_{\bfj}^*)$ is identically distributed as $(A_{\bfi}, A_{\bfj})$, as desired.
\end{proof}

\begin{lemma}\label{lem:stability-almost-monochromatic}
    With probability $1-O(r^2\veps^{1/2})$, both $A_{\bfi}A_{\bfj}$ and $A_{\bfi}^*A_{\bfj}^*$ are edges in $G$ with the same color.
\end{lemma}
\begin{proof}
    It suffices to show that with probability $1-O(r^2\veps^{1/2})$, we have that both $A_{\bfi}A_{\bfj}$ and $A_{\bfi}^*A_{\bfj}$ are edges in $G$ with the same color.
    This is because that $(A_{\bfi},A_{\bfi}^*,A_{\bfj})$ is identically distributed as $(A_{\bfj}, A_{\bfj}^*, A_{\bfi}^*)$ and we can simply apply the union bound.
    Notice that $(X_1,Y_1,X_2)$ always satisfies that $X_1X_2$ and $Y_1X_2$ are edges with the same color.
    Therefore it suffices to show that $(X_1,Y_1,X_2)$ is $O(r^2\veps^{1/2})$-close to $(A_{\bfi}, A_{\bfi}^*, A_{\bfj})$.
    This now follows from \cref{lem:X1-Y1-approx-distribution,lem:X1-X2-approx-distribution} and \cref{prop:total-variation-gluing}.
\end{proof}
Now for each $i\neq j\in [(r-1)^2]$, let $c_{ij}\in[(r-1)^2]$ be the most common color for $A_iA_j$ (we break tie arbitrarily).
\cref{lem:stability-almost-monochromatic} suggests that $A_{\bfi}A_{\bfj}$ should most of the time be an edge with color $c_{\bfi\bfj}$.
For simplicity, set $\chi(uv)=*$ if $uv$ is not an edge in $G$.
We make the previous intuition rigorous as follows.

\begin{lemma}\label{lem:edge-is-prescribed-color}
    We have $\PP(\chi(A_{\bfi}A_{\bfj})\neq c_{\bfi\bfj}) = O(r^2\veps^{1/2}).$
\end{lemma}
\begin{proof}
    Fix any $i\neq j\in [(r-1)^2]$. 
    Then it is clear that
\begin{align*}&1-\PP(A_iA_j, A_i^* A_j^* \textup{ are edges with the same color})\\
\geq& \PP(A_iA_j\textup{ is not an edge})+\sum_{c\in [r]}\PP(\chi(A_iA_j)=c)\PP(\chi(A_i^*A_j^*)\neq c\mid \chi(A_iA_j)=c)\\
\geq& \left(\PP(A_iA_j\textup{ is not an edge})+\sum_{c\in [r]}\PP(\chi(A_iA_j)=c)\right)\PP(\chi(A_i^*A_j^*)\neq c_{ij})\quad(\textup{by the definition of }c_{ij})\\
=&\PP(\chi(A_iA_j)\neq c_{ij}).
\end{align*}
Taking the average over all pairs of $(i,j)$ and applying \cref{lem:stability-almost-monochromatic} gives the desired result.
\end{proof}

In some sense, we have proven that $G$ looks like some blowup of $H$, where $H$ is complete on $[(r-1)^2]$ where the edge $ij$ is colored with $c_{ij}$.
In fact, with this choice of $H$, the second item of \cref{thm:strong-stability} already follows by \cref{lem:edge-is-prescribed-color}.
To see this, simply note that for any $i\neq j\in [(r-1)^2]$,
\[\PP(\chi(A_iA_j)\neq c_{ij})\leq \frac{\PP(\chi(A_{\bfi}A_{\bfj})\neq c_{\bfi\bfj})}{\PP((\bfi,\bfj)=(i,j))} = O(r^6\veps^{1/2}).\]
It remains to prove the last two items in the statement of \cref{thm:strong-stability}.

Previously, we used the strong fact that the partition $V_1\cup\cdots\cup V_{(r-1)^2}$ does not depend on the rainbow triangle with which we construct it to conclude that $H$ is a graph induced by an affine plane of order $(r-1)$.
Now, the best we can hope for is an approximate version of this fact, which significantly complicates the argument.
Instead, we focus on studying the distribution of $(X_1,X_2,X_3)$, which eventually will give us structure for $H$ as a byproduct.

The distribution of $(X_1,X_2,X_3)$ is the same as that of $(X_1,X_2,Y_3)$.
Similar to how we studied the distribution of $(X_1,Y')$, we will first study the distributions of $(X_1,X_2,Y')$ and $(X_1,X_2,Y_2)$.
Then we will deduce the distribution of $(X_1,X_2,Y_3)$ from the fact that $(X_1,X_2,Y')$ is a mixture of $(X_1,X_2,Y_2)$ and $(X_1,X_2,Y_3)$.

We begin with the easier task of investigating the distribution of $(X_1,X_2,Y')$.
Let $\bfj^*$ be a conditionally independent resample of $\bfj$ given $\bfi$.
\begin{lemma}\label{lem:X1-X2-Yprime-approx-distribution}
    The tuple $(X_1,X_2,Y')$ is $O(r^2\veps^{1/2})$-close to $(A_{\bfi}, A_{\bfj}, A_{\bfj^*}^*)$.
\end{lemma}
\begin{proof}
    From \cref{lem:Yprime-approx-independence}, we know that $(X_1,X_2,Y')$ is $O(\veps^{1/2})$-close to $(X_1,X_2,Z')$.
    Recall that $Z'$ is the conditionally independent resample of $Y'$ given $X_1$, and $(X_1,Y')$ is identically distributed as $(X_1,X_2)$.
    Therefore $Z'$ is in fact a conditionally independent resample of $X_2$ given $X_1$.
    Now the lemma follows from \cref{lem:X1-X2-approx-distribution} and \cref{prop:total-variation-gluing}.
\end{proof}

Before we can get to the distribution of $(X_1,X_2,Y_2)$, we need several lemmas.
\begin{lemma}\label{lem:X1-X2-color-approx-distribution}
    We have $(X_1,X_2,\chi(X_1X_2))$ is $O(r^2\veps^{1/2})$-close to $(A_{\bfi}, A_{\bfj}, c_{\bfi\bfj})$.
\end{lemma}
\begin{proof}
    This is an immediate corollary of \cref{lem:X1-X2-approx-distribution} and \cref{lem:edge-is-prescribed-color}.
\end{proof}

Now let $Z_2$ be a conditionally independent resample of $Y_2$ given $X_1,\chi(X_1X_2)$.

\begin{lemma}\label{lem:Y2-approx-independence}
    The tuples $(X_1,\chi(X_1X_2),X_2,X_3,Y_2)$ and $(X_1,\chi(X_1X_2),X_2,X_3,Z_2)$ are $O(r^{1/2}\veps^{1/2})$-close to each other.
\end{lemma}
\begin{proof}
    We know from the last part of \cref{lem:entropy_input} that 
    \begin{align*}\HH(X_1,\chi(X_1X_2),X_2,X_3,Y_2) =& \HH(X_1,X_2,X_3,Y_2) \\\geq &\HH(X_1,X_2,X_3)+\HH(X_2\mid X_1,\chi(X_1X_2))-O(r\veps)\\ = &\HH(X_1,\chi(X_1X_2),X_2,X_3,Z_2)-O(r\veps).
    \end{align*}
    The statement thus follows from \cref{cor:mutual-info-small}.
\end{proof}

We can finally describe the distribution of $X_1,X_2,Y_2$.
Let $\bfj'$ be a conditionally independent resample of $\bfj$ given $\bfi$ and $c_{\bfi\bfj}$.

\begin{lemma}\label{lem:X1-X2-Y2-approx-distribution}
    We have $(X_1,X_2,Y_2)$ and $(A_{\bfi}, A_{\bfj}, A_{\bfj'}^*)$ are $O(r^2\veps^{1/2})$-close to each other.
\end{lemma}
\begin{proof}
    By \cref{lem:Y2-approx-independence}, we know that  the tuple $(X_1,\chi(X_1X_2),X_2,Y_2)$ is $O(r^{1/2}\veps^{1/2})$-close to the tuple $(X_1,\chi(X_1X_2), X_2, Z_2)$.
    By the definition of $Z_2$ and \cref{lem:X1-X2-color-approx-distribution}, we know that $(X_1,\chi(X_1X_2),X_2)$ and $(X_1,\chi(X_1X_2),Z_2)$ are both $O(r^2\veps^{1/2})$-close to $(A_{\bfi}, c_{\bfi\bfj}, A_{\bfj})$, which is identically distributed as $(A_{\bfi}, c_{\bfi\bfj}, A_{\bfj'}^*)$.
    As $X_2$ and $Z_2$ are conditionally independent given $(X_1,\chi(X_1X_2))$, and $A_{\bfj}, A_{\bfj'}^*$ are conditionally independent given $(A_{\bfi}, c_{\bfi\bfj})$, we see that $(X_1,\chi(X_1X_2),X_2,Z_2)$ is $O(r^2\veps^{1/2})$-close to $(A_{\bfi}, c_{\bfi\bfj}, A_{\bfj}, A_{\bfj'}^*)$.
    The desired statement then follows.
\end{proof}

Having understood the distribution of $(X_1,X_2,Y')$ and $(X_1,X_2,Y_2)$, we are almost ready to study the distribution of $(X_1,X_2,X_3)$.
For technical reasons, we need to first show that $H$ is a regular graph in each color class.

\begin{lemma}\label{lem:color-regular}
    For any $c\in[r]$ and any $i\in [(r-1)^2]$, there are exactly $(r-2)$ indices $j\in [(r-1)^2]$ with $j\neq i$ and $c_{ij}=c$.
\end{lemma}
\begin{proof}
    To show this, we will show that $(A_{\bfi}, c_{\bfi\bfj})$ is close to $(A_{\bfi}, \bfc)$ where $\bfc$ is independently and uniformly distributed among $[r]$.
    First, by \cref{lem:X1-X2-color-approx-distribution}, we know that $(A_{\bfi}, c_{\bfi\bfj})$ is $O(r^2\veps^{1/2})$-close to $(X_1,\chi(X_1X_2))$.

    Now by \cref{lem:entropy_input} and \cref{cor:almost-uniform}, we know that $(X_1,\chi(X_1X_2))$ is $O(r^{1/2}\veps^{1/2})$-close to $(X_1,\bfc)$.
    As an immediate corollary, we know that $(A_{\bfi}, c_{\bfi\bfj})$ is $O(r^2\veps^{1/2})$-close to $(X_1, \bfc)$.
    Since $\bfc$ is independent of both $A_{\bfi}$ and $X_1$, and $X_1$ is $O(r^2\veps^{1/2})$-close to $A_{\bfi}$, by \cref{prop:total-variation-gluing} we obtain that $(A_{\bfi}, c_{\bfi\bfj})$ is $O(r^2\veps^{1/2})$-close to $(A_{\bfi}, \bfc)$.
    Now for any $i\in [(r-1)^2]$ and $c\in [r]$, let $N_{i,c}$ be the set of indices $j$ with $c_{ij}=c$.
    By explicitly computing the total variation between $(A_{\bfi}, c_{\bfi\bfj})$ and $(A_{\bfi}, \bfc)$, we have
    \[\frac{1}{(r-1)^2}\sum_{i=1}^{(r-1)^2}\sum_{c=1}^{r}\abs{\frac{\abs{N_{i,c}}}{r(r-2)}-\frac{1}{r}}=O(r^2\veps^{1/2}).\]
    If $\abs{N_{i,c}}\neq r-2$ for some choice of $i,c$, then the left hand side is at least $\frac{1}{r(r-1)^2(r-2)} = \Theta(r^4)$, which is a contradiction as $\veps = o(r^{-12})$.
\end{proof}

We can finally describe the distribution of $(X_1,X_2,X_3).$
Let $\bfk$ be chosen conditionally independently given $\bfi,\bfj$ from the indices $[(r-1)^2]\backslash \{\bfi\}$ so that $c_{\bfi\bfk}\neq c_{\bfi\bfj}$.

\begin{lemma}\label{lem:X1-X2-X3-approx-distribution}
    The random rainbow triangle $(X_1,X_2,X_3)$ is $O(r^2\veps^{1/2})$-close to $(A_{\bfi}, A_{\bfj}, A_{\bfk})$.
\end{lemma}
\begin{proof}
    Observe that $(X_1,X_2,X_3)$ is identically distributed as $(X_1,X_2,Y_3)$.
    Also note that $(A_{\bfi}, A_{\bfj}, A_{\bfk})$ is identically distributed as $(A_{\bfi}, A_{\bfj}, A_{\bfk}^*)$ as $\bfk\neq \bfi,\bfj$ always holds.
    Therefore we just need to show that $(X_1,X_2,Y_3)$ is close to $(A_{\bfi}, A_{\bfj}, A_{\bfk}^*)$.
    
    Note that $(X_1,X_2,Y')$ is the mixture of $(X_1,X_2,Y_2)$ and $(X_1,X_2,Y_3)$ with probability distribution $(\frac{p_2}{p_2+p_3}, \frac{p_3}{p_2+p_3})$.
    On the other hands, $(A_{\bfi}, A_{\bfj}, A_{\bfj^*}^*)$ is the mixture of $(A_{\bfi}, A_{\bfj}, A_{\bfj'}^*)$ and $(A_{\bfi}, A_{\bfj}, A_{\bfk}^*)$ with probabilities $(\frac{1}{r},\frac{r-1}{r})$ (here, we use \cref{lem:color-regular} crucially).
    Now by combining \cref{lem:X1-X2-Yprime-approx-distribution,lem:X1-X2-Y2-approx-distribution} and the fact that $(\frac{p_2}{p_2+p_3}, \frac{p_3}{p_2+p_2})$ is $O(\veps)$-close to $(\frac{1}{r}, \frac{r-1}{r})$, we may apply \cref{cor:advanced-unmixing-variable} to deduce that $(X_1,X_2,Y_3)$ is $O(r^2\veps^{1/2})$-close to $(A_{\bfi}, A_{\bfj}, A_{\bfk}^*)$, as desired.
\end{proof}

This is almost the last item of \cref{thm:strong-stability}.
The only thing we need to make sure is that the distribution of $(\bfi, \bfj, \bfk)$ is the same as sampling vertices of a rainbow triangle in $H$ uniformly at random.
Notice that in order for it to be true, we need the following property of $H$: if $c_{ij}\neq c_{ik}$, then $ijk$ is a rainbow triangle in $H$.
We prove this in the lemma below.

\begin{lemma}\label{lem:rainbow-cherry-to-triangle}
    Recall that $\veps<\alpha r^{-16}$ where $\alpha>0$ is a small constant.
    For sufficiently small $\alpha$, the graph $H$ satisfies that if $i,j,k$ are distinct indices in $[(r-1)^2]$ with $c_{ij}\neq c_{ik}$, then $ijk$ is a rainbow triangle in $H$.
\end{lemma}
\begin{proof}
    Fix any distinct $i,j,k$ with $c_{ij}\neq c_{ik}$.
    Notice that  we have 
    \[\PP((\bfi, \bfj, \bfk)=(i,j,k)) = \frac{1}{(r-1)^2\cdot r(r-2)\cdot (r-1)(r-2)} = \Theta(r^{-6}).\]
    By \cref{lem:X1-X2-X3-approx-distribution}, the probability that $A_i, A_j, A_k$ does not form a rainbow triangle is $O(r^8\veps^{1/2})$.
    Moreover, by the second item of \cref{thm:strong-stability} that we have proven, we know that the event $(\chi(A_iA_j),\chi(A_iA_k), \chi(A_jA_k)) \neq (c_{ij},c_{ik}, c_{jk})$ happens with probability $O(r^6\veps^{1/2})$.
    Therefore the above two failure probabilities add up to $O(r^8\veps^{1/2}) = O(\alpha^{1/2})$, and as long as $\alpha$ is sufficiently small, this is at most $1/2$.
    This shows that there is a rainbow triangle whose edges have colors $c_{ij}, c_{ik}, c_{jk}$.
    In particular, $ijk$ forms a rainbow triangle in $H$.
\end{proof}

Now we may use \cref{lem:color-regular}, \cref{lem:rainbow-cherry-to-triangle} to verify the condition for \cref{lem:extract-affine-plane}.
This shows that $H$ is a graph induced by an affine plane of order $(r-1)$, proving the third item in \cref{thm:strong-stability}.
Finally, to show the fourth item, it suffices to show that $\bfi, \bfj, \bfk$ form a uniform rainbow triangle in $H$.
By \cref{lem:rainbow-cherry-to-triangle}, it suffices to show that $\bfi, \bfj, \bfk$ is chosen uniformly from the triplets $(i,j,k)$ with $c_{ij}\neq c_{ik}$.
This follows from the definition of the random variables and \cref{lem:color-regular}.
We have thus concluded the proof of \cref{thm:strong-stability}.

\section{Concluding remarks}\label{sec:remark}

\subsection{Many colors with a given number of edges for each color}
The number of rainbow triangles when there are only three colors can be bounded more precisely by \cref{thm:rainbow-triangle} if the number of edges of each color is known to us.
It is natural to ask whether there is a better bound for many colors as well if we know this extra piece of information.
In fact, it is possible to mimic the strategy used at the end of \cref{subsec:two_tricks}---we just need to upper bound $\HH(X_1,X_2)$ given the number of edges for each color.
However, when each color has approximately the same number of edges, this optimization produces the exact same bound as \cref{thm:r-color}.
The extra optimization only improves the bound if one or two colors constitute most edges in the graph.
For example, if the numbers of edges are $m_1\geq \cdots \geq m_r$ and $m_1\geq \frac{m}{3}$, $m_2\geq \frac{m-m_1}{2}$, then $m$ can be replaced with $3\sqrt[3]{m_1m_2\left(m_3+\cdots+m_r\right)}$.

Let us also remark that the na\"ive generalization of the $\sqrt{2m_rm_gm_b}$ bound is simply not true for more colors.
Indeed, in the setting of $r$ colors with $r\geq 4$, it is impossible to bound the number of rainbow triangles by $C(m_1\cdots m_r)^{\frac{3}{2r}}$ for any constant $C$, as there can still be rainbow triangles when $m_r=0$.
What the precise answer should be is unclear to us.
Perhaps the first step is to go back to $r=3$ and determine a tighter upper bound when $m_r,m_g,m_b$ are not equal to each other.

\begin{question}
    For any $\alpha_r,\alpha_g,\alpha_b\in(0,1)$ with $\alpha_r+\alpha_g+\alpha_b=1$, determine the smallest constant $C=C(\alpha_r,\alpha_g,\alpha_b)$ so that the following is true.
    In a simple graph $G$ with at most $\alpha_rm$ red edges, at most $\alpha_gm$ green edges and at most $\alpha_bm$ blue edges, there are at most $(C+o(1))(\alpha_r\alpha_g\alpha_b)^{1/2}m^{3/2}$ rainbow triangles.
\end{question}
We know that $C\leq \sqrt{2}$ for any $\alpha_r,\alpha_g,\alpha_b$ with equality when $\alpha_r=\alpha_g=\alpha_b=\frac{1}{3}$.
On the other hand, even a blowup of a rainbow triangle gives $C\geq 1$, and one should be able to obtain a better bound from unbalanced blowups of a properly colored $K_4$.

\subsection{Many colors with a given number of vertices}
An old problem of Erd\H{o}s and S\'os asks for the maximum number of rainbow triangles in a simple $3$-edge-colored graph with a given number of vertices (see \cite{EH72}).
This was resolved exactly for sufficiently large $n$ by Balogh et.~al.~\cite{BHLPVY17} using flag algebras, where they show that balanced iterated blowups of a properly colored $K_4$ maximize the number of rainbow triangles.
Asking the same question for more colors leads to the following natural conjecture.

\begin{conjecture}
    Let $r\geq 3$ be a positive integer.
    Then a simple $r$-edge-colored graph $G$ with $n$ vertices contains at most $\frac{(r-1)(r-2)}{r^2-2r+2}\cdot \frac{n^3}{6}$ rainbow triangles.
\end{conjecture}
If this is true, then it would be asymptotically tight for any $r$ where there is a projective plane of order $(r-1)$ as we can take the graph induced by an affine plane of order $(r-1)$ and consider its balanced iterated blowups.

\subsection{Stability for \texorpdfstring{\cref{thm:properly-colored-Kd-in-r-colors}}{Theorem 1.8}}
As discussed in the introduction, \cref{thm:properly-colored-Kd-in-r-colors} gives, for each $d\geq 4$ and odd $r$, a tight upper bound on the number of properly colored $K_d$'s in a simple $r$-edge-colored graph with $m$ edges.
A natural question is to ask for the best possible bound for even $r$.
It is likely that we could carry out a stability argument akin to the one in \cref{sec:sturctural_stability} to improve the constant by a little bit when $r$ is even, but it is unclear to us what the right constant should be in that case.
To be more specific, we pose the following question.
\begin{question}
    Determine the smallest constant $C$ so that the following holds.
    Every $4$-edge-colored simple graph with $m$ edges contains at most $Cm^2$ properly colored $K_4$'s.
\end{question}

\subsection{Rainbow cliques}\label{subsec:rainbow_Kd}
\cref{thm:properly-colored-Kd-in-r-colors} can be seen as an extension of \cref{thm:r-color} where rainbow triangles are thought of as properly colored triangles.
However, a more natural question is to determine the maximum number of rainbow $K_d$'s in a simple $r$-edge-colored graph with a given number of edges.
This question looks much harder in general.
Take $K_4$ for example.
Below, let $m$ be the number of edges in the graph we are going to construct, and let $N_4$ be the number of rainbow $K_4$'s in the graph.
The question is to maximize $N_4/m^2$ for a given $r$.
For our convenience, we will compute $(24N_4)/(2m)^2$ instead.

Based on our results in this paper, the natural families of graphs to test are properly colored cliques and graphs induced by affine planes.
For $r$ such that $r-1$ is an odd prime power, the graph induced by the affine plane over $\FF_{r-1}$ gives $2m = (r-1)^2\cdot r(r-2)$ and $24N_4 = (r-1)^2\cdot r(r-2)\cdot (r-1)(r-2)\cdot \left[(r-1)^2-(6r-15)\right] = r(r-1)^3(r-2)^2(r-4)^2$, giving
\[\frac{24N_4}{(2m)^2} = \frac{(r-4)^2}{r(r-1)}.\]
If $r-1$ is an even prime power, the same construction gives $2m = (r-1)^2\cdot r(r-2)$ and $24N_4 = (r-1)^2\cdot r(r-2)\cdot (r-1)(r-2)\cdot \left[(r-1)^2-(6r-14)\right] = r(r-1)^3(r-2)^2(r-3)(r-5)$, yielding
\[\frac{24N_4}{(2m)^2} = \frac{(r-3)(r-5)}{r(r-1)}.\]

For $r$ odd, we would also like to consider a properly colored $K_{r+1}$ with $r$ colors.
However, there is not a unique such coloring, and different colorings can give different numbers of rainbow $K_4$'s. In particular, every $K_4$ in a properly colored $K_{r+1}$ is properly colored, and the only obstruction to it being rainbow is that opposite edges could have the same color.
To maximize the number of $K_4$'s, we would thus like to take a proper coloring with the following property: if $abcd$ are four vertices with $\chi(ab)=\chi(cd)$, then $\chi(ad)=\chi(bc)$ and $\chi(ac)=\chi(bd)$ as well.
For example, if $r=2^k-1$, identifying the vertex set of $K_{r+1}$ with $\FF_2^k$ yields a Cayley coloring with this property, where the edge $xy$ receives the color $x+y$.
This graph has $2m = 2^k(2^k-1)= (r+1)r$ and $24N_4 = 2^k(2^k-1)(2^k-2)(2^k-4) = (r+1)r(r-1)(r-3)$, giving 
\[\frac{24N_4}{(2m)^2} = \frac{(r-1)(r-3)}{r(r+1)}.\]
This formula is the largest among all the constructions above.

The situation becomes more interesting for $K_6$.
Let $N_6$ be the number of $K_6$'s. As before, we would like to compute $(6!N_6)/(2m)^3$ to compare different constructions.
For the $\FF_2^k$ construction, we compute that 
\[6!N_6 = 2^k(2^k-1)(2^k-2)(2^k-4)(2^k-8)(2^k-15) = (r+1)r(r-1)(r-3)(r-7)(r-14),\]
and so 
\[\frac{6!N_6}{(2m)^3} = \frac{(r-1)(r-3)(r-7)(r-14)}{r^2(r+1)^2}=1-\Theta(r^{-1}).\]
However, we may also consider another construction if $r = \binom{n}{2}$, in which case we could simply take a rainbow $K_n$.
In this case, we have
\[\frac{6!N_6}{(2m)^3} = \frac{(n-2)(n-3)(n-4)(n-5)}{n^2(n-1)^2} = 1-\Theta(r^{-1/2}).\]
Although this is worse as $r$ gets large, when $r=15$, the former construction gives $k=4$ and $6!N_6/(2m)^3 = 7/300$ and the latter construction gives $n=6$ and $6!N_6/(2m)^3 = 8/300$, which is slightly better.
Motivated by those examples, we pose the following question.

\begin{question}
    Is it true that when $r=2^k-1$ is sufficiently large, the Cayley colored graph of $\FF_2^k$ maximizes the number of rainbow $K_4$'s and the number of rainbow $K_6$'s among all $r$-edge-colored simple graphs with the same number of edges?
\end{question}

We briefly remark that Cairncross, Mizgerd and Mubayi \cite{CMM26} have considered a similar problem for $K_6$ where the number of vertices is fixed instead.
Interestingly, when the number of vertices is fixed, balanced blowups of a rainbow $K_6$ perform worse than balanced blowups of the Cayley graph of $\FF_2^4$.
We also remark that in the Cayley graph construction, when $k$ is sufficiently large, the number of ways to extend a rainbow $K_6$ to a rainbow $K_7$ starts to depend on the particular $K_6$, which makes counting rainbow $K_d$'s trickier.
It also suggests that maybe the Cayley graph construction is not optimal for $K_d$'s when $d$ is large.

\subsection{Generalization to hypergraphs}
Instead of generalizing rainbow triangles to rainbow cliques, we can also move from graphs to hypergraphs.
To be more specific, in a simple $r$-edge-colored $d$-uniform hypergraph with $m$ hyperedges, what is the maximum number of rainbow $K_{d+1}^{(d)}$'s?
This was already asked in \cite{CY24kk} for the case $r=d+1$.

In fact, our argument for \cref{thm:r-color} generalizes to higher uniformities to some extent. 
Although the argument still works, the bound obtained from it is no longer tight.
Take $d=3$ and $r=4$ for example.
Our argument can be generalized to show that there are at most $\sqrt[3]{\frac{5}{16}m^4}$ rainbow $K_4^{(3)}$'s.
However, one can argue that for the bound to be tight, the link at each vertex would need to be a tight example for $d=2$ and $r=4$, which has to be a balanced blowup of the $\FF_3^2$ construction by \cref{thm:tight-construction}.
With a little analysis, one sees that this is impossible.

\section*{Acknowledgement}
The first and third authors would like to thank Shagnik Das and Christoph Spiegel for suggesting the problem and some enlightening conversations.
The research of the second author was partially supported by the National Science Foundation under Grant No.~DMS-2424441.
The research of the third author is supported by the Jane Street Graduate Research Fellowship.

\bibliographystyle{plain}
\bibliography{bib.bib}
\end{document}